\documentclass[11pt]{article}
 \usepackage{amsfonts}
 \usepackage{mathrsfs}
 \usepackage{bbm}
 \usepackage{amsthm}
\usepackage{amsmath,amssymb}
\def\MatrixFont{\bf}
\def\VectorFont{\bf}

\newcommand{\mA}{{\MatrixFont A}}
\newcommand{\mB}{{\MatrixFont B}}
\newcommand{\mC}{{\MatrixFont C}}

\newcommand{\mE}{{\MatrixFont E}}
\newcommand{\mF}{{\MatrixFont F}}

\newcommand{\mI}{{\MatrixFont I}}
\newcommand{\mJ}{{\MatrixFont J}}

\newcommand{\mP}{{\MatrixFont P}}
\newcommand{\mQ}{{\MatrixFont Q}}
\newcommand{\mR}{{\MatrixFont R}}
\newcommand{\mS}{{\MatrixFont S}}
\newcommand{\mT}{{\MatrixFont T}}

\newcommand{\mV}{{\MatrixFont V}}
\newcommand{\mW}{{\MatrixFont W}}

\newcommand{\ve}{{\VectorFont e}}

\newcommand{\vu}{{\VectorFont u}}
\newcommand{\vv}{{\VectorFont v}}
\newcommand{\vw}{{\VectorFont w}}
\newcommand{\vx}{{\VectorFont x}}
\newcommand{\vy}{{\VectorFont y}}
\newcommand{\vz}{{\VectorFont z}}

\def\diag{\qopname\relax o{diag}}

\newtheorem{theorem}{Theorem}[section]
\newtheorem{lemma}[theorem]{Lemma}

\theoremstyle{definition}

\theoremstyle{remark}
\newtheorem{remark}[theorem]{Remark}

\theoremstyle{algorithm}
\newtheorem{algorithm}[theorem]{Algorithm}

\theoremstyle{corollary}

\theoremstyle{example}
\newtheorem{example}[theorem]{Example}

\usepackage{amssymb}
\usepackage{epsfig}
\usepackage{times}
\title{Monte Carlo Set-Membership Filtering for Nonlinear Dynamic Systems}

\author{Zhiguo Wang,  Xiaojing Shen\thanks{This work was supported in part by  the open research funds of BACC-STAFDL of China
under Grant No. 2015afdl010,  the special funds of NEDD of China under Grant No. 201314, the NSF No. 61273074£© and the PCSIRT1273. Zhiguo Wang, Xiaojing Shen (corresponding author), Yunmin Zhu and Jianxin Pan are  with Department of Mathematics, Sichuan
University, Chengdu, Sichuan 610064, China. E-mail: wangzg315@126.com, shenxj@scu.edu.cn,  ymzhu@scu.edu.cn, jianxin.pan@manchester.ac.uk.}, ~Yunmin Zhu and Jianxin Pan}

\begin{document}
 \maketitle
\begin{abstract}
When underlying probability density functions of nonlinear dynamic systems are unknown, the filtering problem is known to be a challenging problem.
This paper attempts to make progress on this problem by proposing a new class of filtering methods in bounded noise setting via set-membership theory and Monte Carlo (boundary) sampling technique, called Monte Carlo set-membership filter. The set-membership prediction and measurement update are derived by recent convex optimization methods based on S-procedure and Schur complement. To guarantee the on-line usage, the nonlinear dynamics are linearized about the current estimate and the remainder terms are then bounded by an optimization ellipsoid, which can be described as a semi-infinite optimization problem. In general, it is an analytically intractable problem when dynamic systems are nonlinear.  However, for a typical nonlinear dynamic system in target tracking, we can analytically derive some regular properties for the remainder. Moreover, based on the remainder properties and the inverse function theorem, the semi-infinite optimization problem can be efficiently solved by Monte Carlo boundary sampling technique. Compared with the particle filter, numerical examples show that when the probability density functions of noises are unknown, the performance of the Monte Carlo set-membership filter is better than that of the particle filter.
\end{abstract}

\noindent{\bf keywords:} Nonlinear dynamic systems;  target tracking; set-membership filter; particle filter;  Monte Carlo set-membership filter.

\section{Introduction}\label{sec_1}

Filtering techniques for dynamic systems are widely used in applied fields such as target tracking, signal processing, automatic control, computer vision and economics, just to name a few. The Kalman filter \cite{Kalman60} is well known as the recursive best linear unbiased state estimator, which is clearly established
as a fundamental tool for analyzing and solving a broad class of filtering problems with linear dynamic systems. When dynamic systems are nonlinear, a few well-known generalizations are the extended Kalman filter (EKF), Gaussian sum filters and unscented Kalman filtering (UKF) (see, e.g., \cite{BarShalom-Li-Kirubarajan01,Simon06}). These methods are based on local linear approximations of the nonlinear system where the higher order terms are ignored.

Most recently, researchers have been attracted to a new class of filtering methods based on the sequential Monte Carlo approach for nonlinear and non-Gaussian dynamic systems. Sequential Monte Carlo methods achieve the filtering task by recursively generating weighted Monte Carlo samples of the state variables by importance sampling. The samples and their weights are then used to estimate expectation, covariance and other system characteristics. The earliest two methods is the \emph{particle filter} (also called the bootstrap filter) \cite{Gordon-Salmond-Smith93} and \emph{sequential imputation} for general missing data problems \cite{Kong-Liu-Wong94}. Subsequently, a lot of methods have been developed in different situations. A sequential importance sampling framework \cite{Liu-Chen98} has been proposed to unify and generalize these methods. Monte Carlo filtering techniques have caught the attention of researchers in many different fields. Many excellent results in different situations can be found in, e.g., \cite{Kotecha-Djuric03}, \cite{Crisan-Doucet02}, \cite{Arulampalam-Maskell-Gordon-Clapp02}, \cite{Chen-Wang-Liu00}, \cite{Zheng-Niu-Varshney14}, and references therein.  Most of these methods are based on the assumptions that probability density functions of the state noise and measurement noise are known. When underlying probability density functions (pdf) are unknown, the filtering problem for nonlinear dynamic systems is known to be a difficult problem.

Actually, when the underlying probabilistic assumptions are not realistic (e.g., the main perturbation may be deterministic), it seems more natural to assume that the state noise and measurement noise are unknown but bounded and to characterize
the set of all values of the parameter or state vector that are consistent with this hypothesis \cite{Polyak-Nazin-Durieu-Walter04}. The
set-membership estimation was considered first at end of 1960s and
early 1970s (see \cite{Schweppe68,Bertsekas-Rhodes71}). The idea of propagating bounding ellipsoids (or boxes,
polytopes, simplexes, parallelotopes, and polytopes) for systems
with bounded noises has also been extensively investigated, for example, see recent papers
\cite{Durieu-Walter-Polyak01,
Calafiore-ElGhaoui04,Polyak-Nazin-Durieu-Walter04,Shen-Zhu-Song-Luo11},
the book \cite{Jaulin-Kieffer-Didrit-Walter01}, and
references therein.
Most of these methods concentrate on the linear dynamic systems.

The set-membership filtering for nonlinear dynamic systems is known to be a challenging problem. Based on ellipsoid-bounded, fuzzy-approximated or Lipschitz-like nonlinearities, several results have been made \cite{Shamma-Tu97,Yang-Li10,Morrell-Stirlling03,Wei-Wang-Shen10}. These results assume that the ellipsoid bounds, the coefficients of fuzzy-approximation or Lipschitz constants are known before filtering, which limit them in real-time implementation. For example, for a typical nonlinear dynamic system in a radar, the bounds of the remainder depends on the past estimates so that they cannot be obtained before filtering. As far as we know, \cite{Scholte-Campbell03} develops a nonlinear set-membership filtering which can estimate ellipsoid bounds of nonlinearities in real-time and is capable of being on-line usage, and the filter is called the extended set-membership filter (ESMF). Specifically, the nonlinear dynamics are linearized about the current estimate and the state bounding ellipsoid is relaxed to an outer bounding box by the ellipsoid projection method, the remainder terms are then bounded using interval mathematics \cite{Moore66}, and finally the output interval box is bounded using an outer bounding ellipsoid by  minimizing the volume of the bounding ellipsoid. Moreover, the set-membership filtering algorithm is derived based on the linear set-membership filtering in the earliest work \cite{Schweppe68}. It is not difficult to see that the outer bounding ellipsoids of both the remainder and the state is conservative. The cumulative effect of the conservative bounding ellipsoid at each time step may yield disconvergence of a filtering.  In fact, if the state bounding ellipsoid were not relaxed to an outer bounding box by the ellipsoid projection method and using some recent linear set-membership filtering techniques \cite{ElGhaoui-Calafiore01},  it should be possible to derive the tighter outer bounding ellipsoids for both the remainder and the state of the nonlinear dynamic system. More details will be clarified in Remark \ref{rmk_3} and Figure \ref{fig_08}.

In this paper, when underlying pdfs of nonlinear dynamic systems are unknown,
we attempt to make progress on the corresponding filtering problem in the bounded noise setting. We propose a new class of filtering methods via set-membership estimation theory and Monte Carlo (boundary) sampling technique, denoted by MCSMF. The set-membership prediction and measurement update of MCSMF are derived by recent convex optimization methods based on S-procedure and Schur complement. To guarantee the on-line usage, the nonlinear dynamics are linearized about the current estimate and the remainder terms are then bounded by an ellipsoid, which can be described as a semi-infinite optimization problem. In general, it is an analytically intractable problem when dynamic systems are nonlinear.  However, for a typical nonlinear dynamic system in target tracking, we can analytically derive some regular properties for the remainder. Moreover, based on the remainder properties and the inverse function theorem, we prove that the boundary of  the remainder set must be from the the boundary of a set $\{||\vu_k||\leq1\}$ when we linearize the nonlinear equations by Taylor's Theorem. Thus, when we take samples from the set $\{||\vu_k||\leq1\}$, the samples on the boundary  $\{||\vu_k||=1\}$ are sufficient to derive the outer bounding ellipsoids of the remainder set. The samples in $\{||\vu_k||<1\}$ is not necessary. Therefore, the computation complexity can be reduced much more. Compared with the particle filter and ESMF in \cite{Scholte-Campbell03}, numerical examples show that when the probability density functions of noises are known, the performance of the particle filter is better than that of ESMF and MCSMF. Nevertheless, when the probability density functions of noises are unknown, the performance of MCSMF is better than that of the other two filters.

The rest of the paper is organized as follows. Preliminaries are given in Section \ref{sec_2}. In Section \ref{sec_3}, the prediction step and the  measurement update step of the set-membership filtering for nonlinear dynamic systems are derived by solving an SDP problem based on S-procedure and Schur complement, respectively. In Section \ref{sec_4}, the bounding ellipsoid of the remainder set is described as a semi-infinite optimization problem and the steps of MCSMF is summarized. In Section \ref{sec_5}, for a typical nonlinear dynamic system in target tracking, some regular properties for the remainder is derived. Based on the remainder properties and the inverse function theorem, the semi-infinite optimization problem can be efficiently solved by Monte Carlo boundary sampling technique. In Section  \ref{sec_6}, numerical examples are given and discussed. In
Section \ref{sec_7}, concluding remarks are provided.

\section{Preliminaries}\label{sec_2}
\subsection{Problem formulation}\label{sec_2_1}

We consider a nonlinear dynamic system
\begin{eqnarray}%
\label{Eqpre_1}\vx_{k+1}&=&f_k(\vx_k)+\vw_k,\\[3mm]
\label{Eqpre_2}\vy_k&=&h_k(\vx_k)+\vv_k,
\end{eqnarray}
where $\vx_k\in \mathcal {R}^n$ is the state of system at time $k$; $\vy_k\in \mathcal {R}^{n_1}$ is the measurement.
$f_k(\vx_k)$ and $h_k(\vx_k)$ are nonlinear functions of $\vx_k$,
$\vw_k\in \mathcal {R}^n$ is the uncertain process noise and $\vv_k\in \mathcal {R}^{n_1}$ is the uncertain measurement noise. They are assumed to be confined to specified ellipsoidal sets
\begin{eqnarray}
\nonumber\mW_k&=&\{\vw_k: \vw_k^T\mQ_k^{-1}\vw_k\leq1\}\\
\nonumber\mV_k&=&\{\vv_k: \vv_k^T\mR_k^{-1}\vv_k\leq1\},
\end{eqnarray}
where $\mQ_k$ and $\mR_k$ are the \emph{shape matrix} of the ellipsoids $\mW_k$ and $\mV_k$, respectively, which are known symmetric positive-definite matrices.
Moreover, we assume that when the nonlinear functions are linearized, the remainder terms can be bounded by an ellipsoid. Specifically,
by Taylor's Theorem, $f_k$ and $h_k$ can be linearized to
\begin{eqnarray}%
\label{Eqpre_3} f_k(\hat{\vx}_k+\mE_{f_k}\vu_k)=f_k(\hat{\vx}_k)+\mJ_{f_k}\mE_{f_k}\vu_k+\Delta f_k(\vu_k),\\[3mm]
\label{Eqpre_4} h_k(\hat{\vx}_k+\mE_{h_k}\vu_k)=h_k(\hat{\vx}_k)+\mJ_{h_k}\mE_{h_k}\vu_k+\Delta h_k(\vu_k),
\end{eqnarray}
where $\mJ_{f_k}=\frac{\partial f_k(\vx_k)}{\partial \vx}|_{\hat{\vx}_k}$ and $\mJ_{h_k}=\frac{\partial h_k(\vx_k)}{\partial \vx}|_{\hat{\vx}_k}$ are Jacobian matrices, $\Delta f_k(\vu_k)$ and $\Delta h_k(\vu_k)$ are high-order remainders, which can be bounded in an ellipsoid for all $||\vu_k||\leq1$, respectively, i.e.,
\begin{eqnarray}%
\label{Eqpre_155} \Delta f_k(\vu_k)\in  \mathcal {E}_{f_k} &=&\{\vx\in
R^n:(\vx-\ve_{f_k})^T{(\mP_{f_k})}^{-1}(\vx-\ve_{f_k})\leq1\},\\[3mm]
\label{Eqpre_154}&=&\{\vx\in R^n: \vx=\ve_{f_k}+\mB_{f_k}\Delta_{f_k}, \mP_{f_k}=\mB_{f_k}\mB_{f_k}^T, \parallel\Delta_{f_k}\parallel\leq1\},\\ \label{}
\label{Eqpre_152} \Delta h_k(\vu_k)\in  \mathcal {E}_{h_k} &=&\{\vx\in
R^{n_1}:(\vx-\ve_{h_k})^T{(\mP_{h_k})}^{-1}(\vx-\ve_{h_k})\leq1\},\\[3mm]
\label{Eqpre_153}&=&\{\vx\in R^{n_1}: \vx=\ve_{h_k}+\mB_{h_k}\Delta_{h_k}, \mP_{h_k}=\mB_{h_k}\mB_{h_k}^T, \parallel\Delta_{h_k}\parallel\leq1\},
\end{eqnarray}
where  $\ve_{f_k}$ and $\ve_{h_k}$ are the center of the ellipsoids $\mathcal {E}_{f_k}$ and $\mathcal {E}_{h_k}$, respectively; $\mP_{f_k}$ and $\mP_{h_k}$ are the shape matrices of the ellipsoids $\mathcal {E}_{f_k}$ and $\mathcal {E}_{h_k}$ respectively.
Note that we do not assume that the ellipsoids  $\mathcal {E}_{f_k}$ and $\mathcal {E}_{h_k}$ are given before filtering. Both of them are predicated in real time.

The corresponding set-membership filtering problem can be formulated as follows.
Suppose that the initial state $\vx_0$ belongs to a given bounding ellipsoid:
\begin{eqnarray}
\label{Eqpre_7} \mathcal {E}_0&=&\{\vx\in
R^n:(\vx-\hat{\vx}_0)^T(\mP_0)^{-1}(\vx-\hat{\vx}_0)\leq1\},
\end{eqnarray}
where  $\hat{\vx}_0$ is the center of ellipsoid $\mathcal {E}_0$; $\mP_0$ is the shape matrix of the ellipsoid $\mathcal {E}_0$ which is a known symmetric positive-definite matrix. At time $k$, given that $\vx_k$ belongs to a current bounding ellipsoid:
\begin{eqnarray}
\label{Eqpre_9}  \mathcal {E}_k&=&\{\vx\in
R^n:(\vx-\hat{\vx}_k)^T(\mP_k)^{-1}(\vx-\hat{\vx}_k)\leq1\}\\
\label{Eqpre_10}&=&\{\vx\in R^n: \vx=\hat{\vx}_k+\mE_k\vu, \mP_k=\mE_k\mE_k^T, \parallel\vu\parallel\leq1\},
\end{eqnarray}
where  $\hat{\vx}_k$ is the center of ellipsoid $\mathcal {E}_k$; $\mP_k$ is a known symmetric positive-definite matrix.

The goal of the set-membership filtering is to determine a bounding ellipsoid $\mathcal {E}_{k+1}$ based on the measurement $\vy_{k+1}$ at time $k+1$, i.e, look for $\hat{\vx}_{k+1}, \mP_{k+1}$ such that the state $\vx_{k+1}$ belongs to
\begin{eqnarray}
\label{Eqpre_11}\mathcal {E}_{k+1} &=&\{\vx\in
R^n:(\vx-\hat{\vx}_{k+1})^T{(\mP_{k+1})}^{-1}(\vx-\hat{\vx}_{k+1})\leq1\},
\end{eqnarray}
whenever I) $\vx_k$ is in $\mathcal {E}_k$, II) the process and measurement noises $\vw_k, \vv_{k+1}$ are bounded in ellipsoids, i.e. $\vw_k\in\mW_k$, $\vv_{k+1}\in\mV_{k+1}$, and III) the remainders $\Delta f_k(\vu_k)\in  \mathcal {E}_{f_k} $  and $\Delta h_k(\vu_k)\in  \mathcal {E}_{h_k}$.
The key problem is how to determine the bounding ellipsoids $\mathcal {E}_{f_k}$  and $\mathcal {E}_{h_k}$ in real-time so that the filtering algorithm can be on-line usage.

Moreover, we provide a state estimation ellipsoid  by minimizing its ``size" which is a function of the shape matrix
$P$ and is denoted by $f(P)$. It is well known that $tr(P)$ corresponds to the sum of squares of semiaxes lengths
of the ellipsoid, and $logdet(P)$ is related to the volume of the ellipsoid. More discussion on  size of the ellipsoid can be seen in \cite{Shen-Zhu-Song-Luo11}.

\section{Set-membership prediction and measurement update}\label{sec_3}
In this section, we derive the prediction step and the measurement step of the set-membership filtering.
Both of them can be converted to solve an SDP problem based on S-procedure and Schur complement. The main results are summarized to  Theorems 1-2. The proofs are given in Appendix.

\subsection{Prediction step}

\begin{theorem}\label{thm_3}
At time $k+1$, based on measurements $\vy_{k}$, the bounding ellipsoids $\mathcal {E}_{f_k}$ and $\mathcal {E}_{h_k}$,
a predicted bounding ellipsoid $ \mathcal
{E}_{k+1|k}=\{\vx:(\vx-\hat{\vx}_{k+1|k})^T(\mP_{k+1|k})^{-1}(\vx-\hat{\vx}_{k+1|k})\leq1\}$
can be obtained by solving the optimization problem in the variables
$\mP_{k+1|k}$, $\hat{\vx}_{k+1|k}$,  nonnegative auxiliary variables $\tau^u\geq0, \tau^w\geq0,
\tau^v\geq0,\tau^f\geq0, \tau^h\geq0$,
\begin{eqnarray}
\label{Eqpre_101} &&\min~~ f(\mP_{k+1|k}) \\[5mm]
\label{Eqpre_102}&&~~\mbox{subject to}~~ -\tau^u\leq0,~ -\tau^w\leq0, ~-\tau^v\leq0,
 -\tau^f\leq0,~ -\tau^h\leq0,\\[5mm]
\label{Eqpre_103} && -\mP_{k+1|k}\prec0,\\[5mm]
\label{Eqpre_104}&&\left[\begin{array}{cc}
    -\mP_{k+1|k}&\Phi_{k+1|k}(\hat{\vx}_{k+1|k})(\Psi_{k+1|k}(\vy_{k}))_{\bot}\\[3mm]
    (\Phi_{k+1|k}(\hat{\vx}_{k+1|k})(\Psi_{k+1|k}(\vy_{k}))_{\bot})^T& ~~-(\Psi_{k+1|k}(\vy_{k}))_{\bot}^T\Xi(\Psi_{k+1|k}(\vy_{k}))_{\bot}\\
\end{array}\right]\preceq0,
\end{eqnarray}
where
\begin{eqnarray}
 \label{Eqpre_105} \Phi_{k+1|k}(\hat{\vx}_{k+1|k})&=&[f_k(\hat{\vx}_{k})+\ve_{f_k}-\hat{\vx}_{k+1|k},~\mJ_{f_k}\mE_{k},
 \mI, ~0, ~\mB_{f_k}, ~0], ~~0\in\mathcal {R}^{n,n_1},\\[3mm]
 \label{Eqpre_106}\Psi_{k+1|k}(\vy_{k})&=& [h_{k}(\hat{\vx}_{k})+\ve_{h_{k}}-\vy_{k}, ~\mJ_{h_{k}}\mE_k, ~0,
~\mI, ~0, ~\mB_{h_{k}}].
\end{eqnarray}
$(\Psi_{k+1|k}(\vy_{k}))_{\bot}$ is the orthogonal complement of $\Psi_{k+1|k}(\vy_{k})$.
 $\mE_{k}$ is the Cholesky factorization of $\mP_{k}$, i.e, $\mP_{k}=\mE_{k}(\mE_{k})^T$. $\ve_{f_k}$, $\ve_{h_k}$, $\mB_{f_k}$, $\mB_{h_k}$ are denoted by (\ref{Eqpre_154}) and (\ref{Eqpre_153}), respectively. $\mJ_{f_k}=\frac{\partial f_k(\vx_k)}{\partial \vx}|_{\hat{\vx}_k}$ and $\mJ_{h_k}=\frac{\partial h_k(\vx_k)}{\partial \vx}|_{\hat{\vx}_k}$.
\begin{eqnarray}
  \label{Eqpre_107}\Xi &=&\diag(1-\tau^u-\tau^w-\tau^v-\tau^f-\tau^h,\tau^uI,\tau^w\mQ_k^{-1}, \tau^v\mR_k^{-1},\tau^fI,\tau^hI).
\end{eqnarray}
\end{theorem}
\textbf{Proof:} See Appendix.

\subsection{Measurement update step}
\begin{theorem}\label{thm_4}
At time $k+1$, based on measurements $\vy_{k+1}$, the predicted bounding ellipsoid $ \mathcal{E}_{k+1|k}$ and the bounding ellipsoid $\mathcal {E}_{h_{k+1}}$,
a bounding ellipsoid $ \mathcal
{E}_{k+1}=\{\vx:(\vx-\hat{\vx}_{k+1})^T(\mP_{k+1})^{-1}(\vx-\hat{\vx}_{k+1})\leq1\}$
can be obtained by solving the optimization problem in the variables
$\mP_{k+1}$, $\hat{\vx}_{k+1}$,  nonnegative auxiliary variables $\tau^u\geq0,
\tau^v\geq0, \tau^h\geq0$,
\begin{eqnarray}
\label{Eqpre_129} &&\min~~ f(\mP_{k+1}) \\[5mm]
\label{Eqpre_130} &&~~\mbox{subject to}~~ -\tau^u\leq0,~-\tau^v\leq0,~ -\tau^h\leq0,\\[5mm]
\label{Eqpre_131} && -\mP_{k+1}\prec0,\\[5mm]
\label{Eqpre_132}&&\left[\begin{array}{cc}
    -\mP_{k+1}&\Phi_{k+1}(\hat{\vx}_{k+1})(\Psi_{k+1}(\vy_{k+1}))_{\bot}\\[3mm]
    (\Phi_{k+1}(\hat{\vx}_{k+1})(\Psi_{k+1}(\vy_{k+1}))_{\bot})^T& ~~-(\Psi_{k+1}(\vy_{k+1}))_{\bot}^T\Xi(\Psi_{k+1}(\vy_{k+1}))_{\bot}\\
\end{array}\right]\preceq0,
\end{eqnarray}
where
\begin{eqnarray}
\label{Eqpre_133}  \Phi_{k+1}(\hat{\vx}_{k+1})&=&[\hat{\vx}_{k+1|k}-\hat{\vx}_{k+1},
 \mE_{k+1|k}, ~0, ~0], ~~0\in\mathcal {R}^{n,n_1}, \\[3mm]
\label{Eqpre_134} \Psi_{k+1}(\vy_{k+1})&=& [h_{k+1}(\hat{\vx}_{k+1|k})+\ve_{h_{k+1}}-\vy_{k+1},
\mJ_{h_{k+1|k}}\mE_{k+1|k}, ~\mI, ~\mB_{h_{k+1}}].
\end{eqnarray}
$(\Psi_{k+1}(\vy_{k+1}))_{\bot}$ is the orthogonal complement of $\Psi_{k+1}(\vy_{k+1})$.
 $\mE_{k+1|k}$ is the Cholesky factorization of $\mP_{k+1|k}$, i.e, $\mP_{k+1|k}=\mE_{k+1|k}(\mE_{k+1|k})^T$. $\hat{\vx}_{k+1|k}$ is the center of the predicted bounding ellipsoid $ \mathcal{E}_{k+1|k}$. $\ve_{h_{k+1}}$ and $\mB_{h_{k+1}}$ are denoted by $(\ref{Eqpre_153})$ at the time step $k+1$.
 $\mJ_{h_{k+1|k}}=\frac{\partial h_{k+1}(\vx_k)}{\partial \vx}|_{\hat{\vx}_{k+1|k}}$.
\begin{eqnarray}
\label{Eqpre_135} \Xi =\diag(1-\tau^u-\tau^v-\tau^h,\tau^u\mI,\tau^v\mR_{k+1}^{-1},\tau^hI).
\end{eqnarray}
\end{theorem}
\textbf{Proof:} See Appendix.
\begin{remark}
Notice that  if
$f(P)=tr(P)$, the optimization problem in Theorems \ref{thm_3}-\ref{thm_4} is an SDP problem. If
$f(P)=\mbox{logdet(P)}$, it is a MAXDET problem. Both of them can
also be efficiently solved in polynomial-time by interior point
methods for convex programming (see, e.g.,
\cite{Calafiore-ElGhaoui04,Nesterov-Nemirovski94}) and related softwares
\cite{Lofberg04,Sturm99}.
\end{remark}

\section{Monte Carlo Set Membership Filtering}
In this section, we discuss the key problem that how to adaptively determine a bounding ellipsoid to cover the high-order remainders. In the first subsection, for the general case, the problem can be converted to solve a SDP problem via Monte Carlo sampling. Moreover,  the Monte Carlo set membership filtering is presented. In the second subsection, for target tracking, we prove that the remainder can be bounded via Monte Carlo boundary sampling. Thus, the computation complexity  Algorithm \ref{alg_1} can be reduced much more.

\subsection{Ellipsoid bounding of the remainder via Monte Carlo sampling}\label{sec_4}
By (\ref{Eqpre_3})-(\ref{Eqpre_4}), the high-order remainders are
\begin{eqnarray}
 \nonumber \Delta f_k(\vu_k)=f_k(\hat{\vx}_k+\mE_k\vu_k)-f_k(\hat{\vx}_k)-\mJ_{f_k}\mE_k\vu_k,\\
 \nonumber \Delta h_k(\vu_k)=h_k(\hat{\vx}_k+\mE_k\vu_k)-h_k(\hat{\vx}_k)-\mJ_{h_k}\mE_k\vu_k,
\end{eqnarray}
whenever  $\parallel \vu_k\parallel\leq 1$. Obviously, it is a hard problem to cover a remainder by an ellipsoid since $f_k$ and $h_k$ are generally nonlinear functions.
The outer bounding ellipsoid for $\Delta f_k(\vu_k)$ is not uniquely defined, but which can be optimized by minimizing the size $f(P)$ of the bounding ellipsoid.
Thus, the optimization problem for the bounding ellipsoid of $\Delta f_k(\vu_k)$ can be written as
\begin{eqnarray}
\label{Eqpre_48} &&\min~~ f(\mP_{f_k}) \\[5mm]
\label{Eqpre_49} && \mbox{subject to}~ (\Delta f_k(\vu_k)-\ve_{f_k})^T(\mP_{f_k})^{-1}
(\Delta f_k(\vu_k)-\ve_{f_k})\leq 1, \mbox{for~ all}~ ||\vu_k||\leq 1.
\end{eqnarray}
where $\mP_{f_k}=\mB_{f_k}\mB_{f_k}^T$, and $\ve_{f_k}$, $\mP_{f_k}$ are decision variables.
It is called a semi-infinite optimization problem by \cite{Boyd-Vandenberghe04}.

For a general nonlinear dynamic system, to solve the problem (\ref{Eqpre_48}), we may use Monte Carlo  sampling by uniformly taking some samples from the boundary and interior-points of the sphere $||\vu_k||\leq 1$ so that we can get a finite set of  $\vu_k^1,\ldots,\vu_k^N$, then the infinite constraint (\ref{Eqpre_49}) can be approximated by $N$ constraints based on $\vu_k^1,\ldots,\vu_k^N$.
Moreover, by Schur complement, an approximate bounding ellipsoid for $\Delta f_k(\vu_k)$ can be derived by solving the flowing SDP optimization problem:
\begin{eqnarray}
\label{Eqpre_52} &&\min~~ f(\mP_{f_k}) \\[5mm]
\label{Eqpre_53}&&
 \mbox{subject to}~ \left[
                   \begin{array}{cc}
                     -1 & (\Delta f_k(\vu_k^i)-\ve_{f_k})^T \\
                    \Delta f_k(\vu_k^i)-\ve_{f_k} & -\mP_{f_k} \\
                   \end{array}
                 \right]    \prec 0,  i=1,\ldots,N.
\end{eqnarray}

Similarly, the outer bounding ellipsoid for $h_k(\vu_k)$ can be derived by solving
\begin{eqnarray}
\label{Eqpre_59} &&\min~~ f(\mP_{h_{k}}) \\[5mm]
 \label{Eqpre_60} &&
 \mbox{subject to}~ \left[                   \begin{array}{cc}
                     -1 & (\Delta h_{k}(\vu_k^i)-\ve_{h_{k}})^T \\
                    \Delta h_{k}(\vu_k^i)-\ve_{h_{k}} & -\mP_{h_{k}} \\
                   \end{array}
                 \right]  \prec 0,  i=1,\ldots,N.
\end{eqnarray}

\begin{remark}
The problem (\ref{Eqpre_52}) is an SDP problem that can be efficiently solved using modern interior-point methods, which have been developed by \cite{Nesterov-Nemirovski94} and \cite{Vandenberghe-Boyd96}.
When large number of samples are required to guarantee the bounding ellipsoid contain the remainder, the one-order optimizing algorithm \cite{Selin15} may be used for solving the problem (\ref{Eqpre_52})  with a lower computation complexity.
In addition, in the next subsection, we will develop boundary sampling technique for target tracking, where  the samples on boundary are sufficient to derive the outer bounding ellipsoids of the remainder. Thus,  computation complexity can be reduced much more. Numerical examples show that only 50 uniform samples on the boundary  are enough to guarantee the bounding ellipsoid contain the remainder. 
\end{remark}


\begin{remark}\label{rmk_3}
Note that the bounding ellipsoid of \cite{Scholte-Campbell03} is derived by interval mathematics. We  derive the bounding ellipsoid by solving a semi-infinite optimization problem. Figure \ref{fig_08} illustrates the difference of two methods. It is obvious to see that the bounding ellipsoid derived by solving the SDP (\ref{Eqpre_52})  is tighter than that obtained by interval mathematics. The cumulative effect of the conservative bounding ellipsoid at each time step may yield disconvergence of a filtering.
\end{remark}

\begin{figure}[h]
\vbox to 8cm{\vfill \hbox to \hsize{\hfill
\scalebox{0.6}[0.6]{\includegraphics{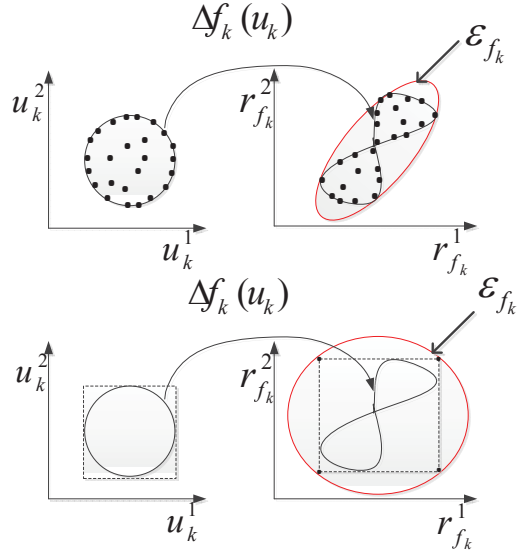}} \hfill}\vfill}
\caption{(top) The bounding ellipsoid is derived by covering the solid points of the remainder which are obtained by Monte Carlo sampling; (bottom) The bounding ellipsoid is derived by covering the vertices of the rectangle obtained by interval mathematics \cite{Scholte-Campbell03}.}\label{fig_08}
\end{figure}

Based on Theorems \ref{thm_3}--\ref{thm_4} and the ellipsoids derived by solving the SDP optimization problems (\ref{Eqpre_52})-(\ref{Eqpre_53}), (\ref{Eqpre_59})-(\ref{Eqpre_60}), the filtering algorithm can be summarized as follows:

\begin{algorithm}[Monte Carlo Set Membership Filtering]\label{alg_1}
~\\
\begin{itemize}
\item Step~1: (Initialization step) Set $k=0$ and  initial values $(\hat{\vx}_0,\mP_0)$ such that $\vx_0\in\mathcal {E}_0$.
\item Step~2: (Bounding step) Take samples $\vu_k^1,\ldots,\vu_k^N$ from the sphere $||\vu_k||\leq 1$, and then determine two bounding ellipsoids to cover the remainders  $\Delta f_k$ and $\Delta h_k$ by (\ref{Eqpre_52})-(\ref{Eqpre_53}) and (\ref{Eqpre_59})-(\ref{Eqpre_60}), respectively.
\item Step~3: (Prediction step) Optimize the center and shape matrix of the state prediction ellipsoid\\ $(\hat{\vx}_{k+1|k},\mP_{k+1|k})$ such that $\vx_{k+1|k}\in\mathcal {E}_{k+1|k}$ by solving the optimization problem (\ref{Eqpre_101})-(\ref{Eqpre_104}).
\item Step~4: (Bounding step) Take samples $\vu_{k+1|k}^1,\ldots,\vu_{k+1|k}^N$ from the sphere $||\vu_{k+1|k}||\leq 1$, and then determine one bounding ellipsoid to cover the remainder  $\Delta h_{k+1|k}$ by  (\ref{Eqpre_59})-(\ref{Eqpre_60}).
 \item Step~5: (Measurement update step) Optimize the center and shape matrix of the state estimation ellipsoid $(\hat{\vx}_{k+1},\mP_{k+1})$ such that $\vx_{k+1}\in\mathcal {E}_{k+1}$ by solving the optimization problem  (\ref{Eqpre_129})-(\ref{Eqpre_132}).
 \item Step~6: Set $k=k+1$ and go to step~2.
\end{itemize}
\end{algorithm}
A flowchart of the Algorithm \ref{alg_1} is given in Figure \ref{fig_09}.

\begin{figure}[h]
\vbox to 8cm{\vfill \hbox to \hsize{\hfill
\scalebox{0.5}[0.45]{\includegraphics{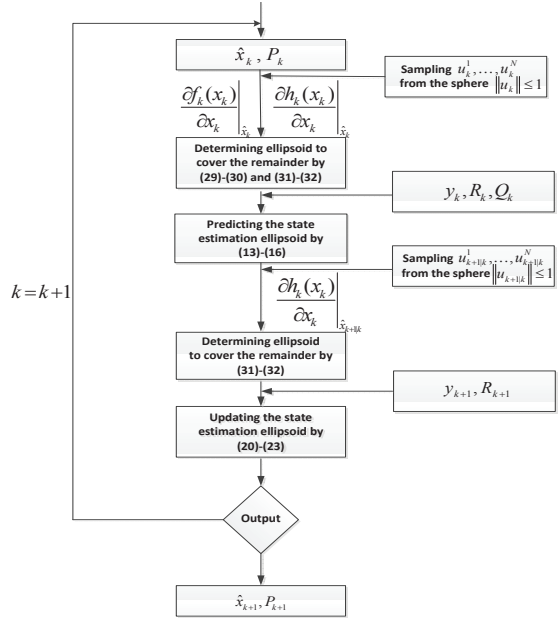}} \hfill}\vfill}
\caption{The flowchart of Algorithm \ref{alg_1}.}\label{fig_09}
\end{figure}

\subsection{Monte Carlo set membership filtering for target tracking}\label{sec_5}

In this subsection, for a typical nonlinear dynamic system in target tracking, we discuss that the remainder can be bounded by an ellipsoid via Monte Carlo \emph{boundary} sampling for target tracking.  We prove that the boundary of  the remainder set $\{\Delta h_{k+1}(\vu_k): ||\vu_k||\leq1\}$ must be from the the boundary of the sphere $\{||\vu_k||\leq1\}$ when we linearize the nonlinear equations by Taylor's Theorem. Thus, when we take samples from the set $\{||\vu_k||\leq1\}$, the samples on the boundary  $\{||\vu_k||=1\}$ are sufficient to derive the outer bounding ellipsoids of the remainder set. Therefore, the computation complexity in the bounding steps of Algorithm \ref{alg_1} can be reduced much more.

Let us consider the following nonlinear measurement equation \cite{BarShalom-Li-Kirubarajan01}:
\begin{eqnarray}
\label{Eq5_1} h(\vx)=\left[
                      \begin{array}{c}
                        \sqrt{(\vx(1)-a)^2+(\vx(2)-b)^2} \\[3mm]
                        arctan\left(\frac{\vx(2)-b}{\vx(1)-a}\right) \\
                      \end{array}
                    \right], a, b \in \mathcal {R}
\end{eqnarray}
where $\vx$ is a four-dimensional state variable that includes
position and velocity $(x, y, \dot{x}, \dot{y})$.
Note that the $h(\vx)$ only depends on the first two dimensions $\vx(1)$ and $\vx(2)$.

We discuss the relationship between the set $\{||\vu_k||\leq 1, \vu_k= [\vu_k(1) ~\vu_k(2)]\}$ and the remainder set $\{\Delta h_{k+1}(\vu_k): ||\vu_k||\leq1\}$.

%
\begin{theorem}\label{thm_5}
If we let the remainder $g(\vu)=h(\vx+\mE\vu)-h(\vx)-\mJ_h\mE\vu$ where h(\vx) is defined in (\ref{Eq5_1}), $\mE$ is a Cholesky factorization of a positive-definite $\mP$ such that $\{\vx+\mE\vu:\parallel u\parallel\leq 1\}$  is not intersect with the radial $\vx(1)<=a, \vx(2)=b$, then
the boundary of the remainder set $\mS=\{g(\vu): \parallel\vu \parallel\leq 1\}$  belongs to the set $\{g(\vu): \parallel \vu \parallel= 1\}$.
\end{theorem}

The proof of Theorem \ref{thm_5} relies on the following three lemmas.

\begin{lemma}[Remainder Lemma]\label{lem_6}
The determinant of the derivative of the remainder $g(\vu)$ is not less than $0$,
and the equality holds if and only if $c\vu(1)+d\vu(2)=0$, where
$c=\mE_{11}(\vx(2)-b)-\mE_{21}(\vx(1)-a)$, $d=\mE_{12}(\vx(2)-b)-\mE_{22}(\vx(1)-a)$, and $\mE_{ij}$ is the entry of the $ith$ row
and the $jth$ column of the matrix $\mE$. Meanwhile, if $c\vu(1)+d\vu(2)=0$, then $g(\vu)=0$.
\end{lemma}
\begin{proof}
From the definition of the function $g(\vu)$, it is easy to see that $g(\vu)$ is a continuously differentiable function. By simple calculations, Jacobian matrix $\mJ_g$ of $g(\vu)$ is
\begin{eqnarray}\nonumber &&\mJ_g=\\
\nonumber&&\left[
                  \begin{array}{cc}
                    \frac{\bigtriangleup_1}{\sqrt{\bigtriangleup_1^2+\bigtriangleup_2^2}}- \frac{\bigtriangledown_1}{\sqrt{\bigtriangledown_1^2+\bigtriangledown_2^2}}
                    &  \frac{\bigtriangleup_2}{\sqrt{\bigtriangleup_1^2+\bigtriangleup_2^2}}- \frac{\bigtriangledown_2}{\sqrt{\bigtriangledown_1^2+\bigtriangledown_2^2}} \\
                     \frac{-\bigtriangleup_2}{\bigtriangleup_1^2+\bigtriangleup_2^2}- \frac{-\bigtriangledown_2}{\bigtriangledown_1^2+\bigtriangledown_2^2}
                     &  \frac{\bigtriangleup_1}{\bigtriangleup_1^2+\bigtriangleup_2^2}- \frac{\bigtriangledown_1}{\bigtriangledown_1^2+\bigtriangledown_2^2} \\
                  \end{array}
                \right]\mE\\[3mm]
\nonumber              && =\mJ_h\mE.
\end{eqnarray}
where
\begin{eqnarray}
\label{Eqpre_156}\bigtriangleup_1&=&\vx(1)+\mE_{11}\vu(1)+\mE_{12}\vu(2)-a,\\
\label{Eqpre_157}\bigtriangleup_2&=&\vx(2)+\mE_{21}\vu(1)+\mE_{22}\vu(2)-b,\\
\label{Eqpre_158}\bigtriangledown_1&=&\vx(1)-a,\\
\label{Eqpre_159}\bigtriangledown_2&=&\vx(2)-b,
\end{eqnarray}
\begin{eqnarray}\nonumber &&\mJ_h=\\
\nonumber&&\left[
                  \begin{array}{cc}
                    \frac{\bigtriangleup_1}{\sqrt{\bigtriangleup_1^2+\bigtriangleup_2^2}}- \frac{\bigtriangledown_1}{\sqrt{\bigtriangledown_1^2+\bigtriangledown_2^2}}
                    &  \frac{\bigtriangleup_2}{\sqrt{\bigtriangleup_1^2+\bigtriangleup_2^2}}- \frac{\bigtriangledown_2}{\sqrt{\bigtriangledown_1^2+\bigtriangledown_2^2}} \\
                     \frac{-\bigtriangleup_2}{\bigtriangleup_1^2+\bigtriangleup_2^2}- \frac{-\bigtriangledown_2}{\bigtriangledown_1^2+\bigtriangledown_2^2}
                     &  \frac{\bigtriangleup_1}{\bigtriangleup_1^2+\bigtriangleup_2^2}- \frac{\bigtriangledown_1}{\bigtriangledown_1^2+\bigtriangledown_2^2} \\
                  \end{array}
                \right].
\end{eqnarray}
Simplifying the determinant of $\mJ_h$,
\begin{eqnarray}
\nonumber &&det(\mJ_h)=\\
\nonumber &&~~~\frac{\left(\sqrt{(\bigtriangleup_1^2+\bigtriangleup_2^2)(\bigtriangledown_1^2+\bigtriangledown_2^2)}
-(\bigtriangleup_1\bigtriangledown_1+\bigtriangleup_2\bigtriangledown_2)\right)
}
{(\bigtriangleup_1^2+\bigtriangleup_2^2)(\bigtriangledown_1^2+\bigtriangledown_2^2)}\\
\nonumber &&~~~\cdot\left(\sqrt{\bigtriangleup_1^2+\bigtriangleup_2^2}+\sqrt{\bigtriangledown_1^2+\bigtriangledown_2^2}\right).
\end{eqnarray}
Thus, $det(\mJ_h)\geq0$ and the equality holds if and only if
$\bigtriangleup_2\bigtriangledown_1=\bigtriangleup_1\bigtriangledown_2$. Since $det(\mJ_g)=det(\mJ_h)det(\mE)$ and $det(\mE)>0$, then $det(\mJ_g)\geq0$ and the equality holds if and only if
$\bigtriangleup_2\bigtriangledown_1=\bigtriangleup_1\bigtriangledown_2$, at the same time, we have $g(\vu)=0$. Moreover, by (\ref{Eqpre_156})-(\ref{Eqpre_159}), it is easy to see that $\bigtriangleup_2\bigtriangledown_1=\bigtriangleup_1\bigtriangledown_2$ is equivalent to $c\vu(1)+d\vu(2)=0$, where
$c=\mE_{11}(\vx(2)-b)-\mE_{21}(\vx(1)-a)$, $d=\mE_{12}(\vx(2)-b)-\mE_{22}(\vx(1)-a)$, and $\mE_{ij}$ is the entry of the $ith$ row
and the $jth$ column of the matrix $\mE$.
\end{proof}

\begin{lemma}\label{lem_5}
If the sets $\mS^1\bigcup\mS^2=\mS^3\bigcup\mS^4$, $\mS^3\bigcap\mS^4=\emptyset$, $\mS^1\subset\mS^3$, then $\mS^4\subset\mS^2$.
\end{lemma}

\begin{proof} 
Since $\mS^1\subset\mS^3$, then $\mS^1\bigcup\mS^2\subset\mS^3\bigcup\mS^2$. Using $\mS^1\bigcup\mS^2=\mS^3\bigcup\mS^4$, we obtain $\mS^3\bigcup\mS^4\subset\mS^3\bigcup\mS^2$, then $(\mS^3\bigcup\mS^4)\bigcap\mS^4\subset(\mS^3\bigcup\mS^2)\bigcap\mS^4$. By $\mS^3\bigcap\mS^4=\emptyset$, we have $\mS^4\subset\mS^4\bigcap\mS^2$.  Moreover, $\mS^4\subset\mS^2$.
\end{proof}

\begin{lemma}[Inverse Function Theorem by \cite{Spivak65}]\label{lem_4}
Suppose that $\varphi:\mR^n\rightarrow\mR^n$ is continuously differentiable in an open set containing $\vu$, and $det(\varphi^{'}(\vu))\neq0$. Then there is an open set $\mV$ containing $\vu$ and open set $\mW$ containing $\varphi(\vu)$ such that
$\varphi: \mV\rightarrow\mW$ has a continuous inverse $\varphi^{-1}: \mW\rightarrow\mV$ which is  differentiable and for all
$\vy\in\mW$ satisfies
\begin{eqnarray}
\nonumber (\varphi^{-1})^{'}(\vy)=[\varphi^{'}(\varphi^{-1}(\vy))]^{-1}.
\end{eqnarray}

\end{lemma}

\begin{proof}[Proof of Theorem \ref{thm_5}] Since $g(\vu)$ is a continuous function in $\mS_1=\{||\vu||\leq1\}$ and $\mS_1$ is compact,  $\mS=\{g(\vu): \parallel\vu \parallel\leq 1\}$is compact \cite{Rosenlicht68}. If we denote the interior and the boundary of the set $\mS$ by $\mS^i$ and $\mS^b$ respectively,
then $\mS=\mS^i\bigcup\mS^b$ and $\mS^i\bigcap\mS^b=\emptyset$. We need to prove $\mS^b\subset\{g(\vu):\parallel \vu \parallel= 1\}$.

By definition of the set $\mS_1$, we can divide it into two parts, i.e., $\mS_1=\mS_1^1\bigcup\mS_1^2$, where
$\mS_1^1=\{\vu:\parallel \vu \parallel= 1~ or~ c\vu(1)+d\vu(2)=0\}$, $\mS_1^2=\{\vu:\parallel \vu \parallel< 1, c\vu(1)+d\vu(2)\neq0\}$ and $c, d$ are defined in Lemma \ref{lem_6}. According to the expression of the set $\mS$, then, we can divide the set $\mS$ into the corresponding parts, i.e., $\mS=\mS^1\bigcup\mS^2$, where $\mS^1=\{g(\vu): \vu\in\mS_1^1\}$ and $\mS^2=\{g(\vu): \vu\in\mS_1^2\}$. Thus $\mS^1\bigcup\mS^2=\mS^i\bigcup\mS^b$.

Next, we prove $\mS^2\subset\mS^i$. For $\forall \vz\in\mS^2$,~$\exists\vu\in\mS_1^2$, s.t., $\vz=g(\vu)$. From the definition of the set $\mS_1^2$, we can see that $det(\mJ_g)>0$ with Lemma \ref{lem_6}. Using Lemma \ref{lem_4}, we can find an open set $\mW\in\mS$ containing $g(\vu)$, in other words, $\vz$ is the interior point of $\mS$, i.e., $\vz\in\mS^i$, thus,
$\mS^2\subset\mS^i$. According to Lemma \ref{lem_5}, we can obtain $\mS^b\subset\mS^1$.

Moreover, we prove that $\mS^1=\{g(\vu):\parallel \vu \parallel= 1\}$. Note that $\mS^1=\{g(\vu):\parallel \vu \parallel= 1\}\bigcup\{g(\vu):c\vu(1)+d\vu(2)=0\}$. According to Lemma \ref{lem_6}, it is obvious that $\{g(\vu):c\vu(1)+d\vu(2)=0\}=\{0\}$. Let $\vu_0=[\frac{-d}{\sqrt{d^2+c^2}}~ \frac{c}{\sqrt{d^2+c^2}}]$, then $\vu_0\in\{\vu:c\vu(1)+d\vu(2)=0\}\bigcap\{\vu:\parallel \vu \parallel= 1\}$, we can also get   $g(\vu_0)\in\{g(\vu):c\vu(1)+d\vu(2)=0\}=\{0\}$ and $g(\vu_0)\in\{g(\vu):\parallel \vu \parallel= 1\}$, then $\{0\}\subset\{g(\vu):\parallel \vu \parallel= 1\}$. Thus, $\mS^1=\{g(\vu):\parallel \vu \parallel= 1\}$.

Therefore, we have $\mS^b\subset\{g(\vu):\parallel \vu \parallel= 1\}$, in other words, the boundary of $\mS=\{g(\vu): \parallel u\parallel\leq 1\}$ belongs to the set $\{g(\vu): \parallel \vu \parallel= 1\}$.
\end{proof}

\begin{example}
To illustrate Theorem \ref{thm_5}, we give an example as follow: if $a=50$, $b=100$, $x=[80 ~130]^T$, $\mP=diag(500,1000)$, it is easy to check that $g(\vu)$ is continuously differentiable in set $\mS_1=\{\vu:\parallel \vu \parallel\leq 1\}$. We divide $\mS_1$ into three parts, i.e., $\mS_1=\mA^1\cup \mB^1\cup\mC^1$, where $\mA^1=\{\vu: c\vu(1)+d\vu(2)<0, \parallel \vu \parallel\leq 1\}$, $\mB^1=\{\vu: c\vu(1)+d\vu(2)>0, \parallel \vu \parallel\leq 1\}$, and $\mC^1=\{\vu: c\vu(1)+d\vu(2)=0, \parallel \vu \parallel\leq 1\}$. Meanwhile, we can also divide $\mS$ into the corresponding parts, such that $\mA=\{g(\vu): \vu\in \mA^1\}$, $\mB=\{g(\vu): \vu\in \mB^1\}$, $\mC=\{g(\vu): \vu\in \mC^1\}$, then $\mS=\mA\cup\mB\cup\mC$.

Figure \ref{fig_01} shows that the separation area of the circle and their corresponding area of $g(\vu)$. Three observations can be seen:
\begin{itemize}
  \item The remainder set is the union of two sets.
  \item The (red) line $\mC^1$ is mapped to the point 0.
  \item The boundary of $\mS$ belongs to the set $\{g(\vu): \parallel \vu \parallel= 1\}$.
\end{itemize}
Thus, when take samples by Monte Carlo methods, the samples on boundary are sufficient to derive the outer bounding ellipsoids of the remainder set. Therefore, based on Theorem \ref{thm_5}, the computation complexity in the bounding steps of Algorithm \ref{alg_1} can be reduced much more.
\end{example}
\begin{figure}[h]
\vbox to 8cm{\vfill \hbox to \hsize{\hfill
\scalebox{0.6}[0.7]{\includegraphics{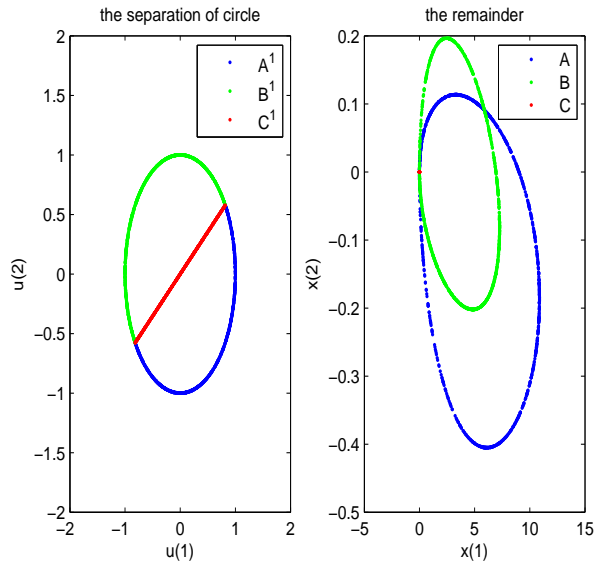}} \hfill}\vfill}
\caption{(left) the separation of circle. (right) the corresponding area of of $g(\vu)$}\label{fig_01}
\end{figure}
\begin{remark}
Note that the assumption that $\mE$ is a Cholesky factorization of a positive-definite $\mP$ such that $\{\vx+\mE\vu:\parallel u\parallel\leq 1\}$  is not intersect with the radial $\vx(1)<=a, \vx(2)=b$ is a weak condition. If the true target is near it, we can transform the data to a new coordinate system where the target far way the the radial, then the assumption can be satisfied.
\end{remark}

\section{Numerical examples in target tracking}\label{sec_6}
In this section, we compare the performance between Monte Carlo set membership filter and particle filter when the underlying probability density functions of noises are known or unknown. Meanwhile, we also compare it with the extended set-membership filter (ESMF) in \cite{Scholte-Campbell03}.

Considering a two-dimensional Cartesian coordinate system, we track a moving target using measured range and angle from one sensor. The system equation is as follows \cite{BarShalom-Li-Kirubarajan01}:
\begin{eqnarray}%
\label{Eqpre_54}\vx_{k+1}&=&f_k(\vx_k)+\vw_k,\\[3mm]
\label{Eqpre_55}\vy_k&=&h_k(\vx_k)+\vv_k,
\end{eqnarray}
where
\begin{eqnarray}%
\nonumber f_k(\vx_k)&=& \left[
                         \begin{array}{cccc}
                           1 & 0& T& 0\\
                           0 & 1 & 0 & T \\
                           0 & 0 & 1 & 1 \\
                           0 & 0 & 0 & 1 \\
                         \end{array}
                       \right]\vx_k
\end{eqnarray}

\begin{eqnarray}
\nonumber h_k(\vx_k)=\left[
                      \begin{array}{c}
                        \sqrt{(\vx_k(1))^2+(\vx_k(2))^2} \\[3mm]
                        arctan\left(\frac{\vx_k(2)}{\vx_k(1)}\right) \\
                      \end{array}
                    \right].
\end{eqnarray}
The $\vx$ is a four-dimensional state variable that includes
position and velocity $(x, y, \dot{x}, \dot{y})$, $T=0.2$s is the time sampling
interval.
The process noise and measurement noise assumed to be confined to specified ellipsoidal sets
\begin{eqnarray}
\nonumber\mW_k&=&\{\vw_k: \vw_k^T\mQ_k^{-1}\vw_k\leq1\}\\
\nonumber\mV_k&=&\{\vv_k: \vv_k^T\mR_k^{-1}\vv_k\leq1\}.
\end{eqnarray}
where
\begin{eqnarray}
\nonumber \mQ_k&=&\sigma^2\left[
                          \begin{array}{cccc}
                            \frac{T^3}{3} & 0 & \frac{T^2}{2} & 0 \\
                            0 & \frac{T^3}{3} & 0 & \frac{T^2}{2} \\
                            \frac{T^2}{2} & 0 & T & 0 \\
                            0 & \frac{T^2}{2} & 0 & T \\
                          \end{array}
                        \right]\\
 \nonumber \mR_k&=&\left[
                   \begin{array}{cc}
                     0.3^2 & 0\\
                     0 & {0.1}^2\\
                   \end{array}
                 \right].
\end{eqnarray}
The target acceleration is $\sigma^2=50$.
In the example, the target starts at the point $(50,30)$ with a velocity of $(5,5)$.
The center and the shape matrix of the initial bounding ellipsoid are $\hat{\vx}_0=\left[
 \begin{array}{cccc}
    49.5& 29.5 & 5 &5\\
 \end{array}
\right]^T$,
\begin{eqnarray}
\nonumber \mP_0=\left[
                  \begin{array}{cccc}
                    5 & 0 & 0 & 0 \\
                    0 & 5 & 0 & 0 \\
                    0 & 0 & 2 & 0 \\
                    0 & 0 & 0 & 2 \\
                  \end{array}
                \right],
\end{eqnarray}
respectively.
Assume that the noises are confined to specified ellipsoidal sets, the state noise is truncated Gaussian with  mean $[-0.2~-0.2~-1~-1]$ and covariance $\mQ_k/3^2$ and measurement noise is truncated Gaussian, with mean $[-0.4~0]^T$, covariance $\mR_k/3^2$ on the ellipsoidal sets, respectively.

From the description of the above, we can see that the condition of Algorithm \ref{alg_1} is satisfied, then, using MCSMF to calculate the error bound, which is defined as follows:
\begin{eqnarray}\label{Eqpre_200}
error(k)=\frac{1}{m}\sum_{i=1}^m|\vx_k^i-\hat{\vx}_k^i|,
\end{eqnarray}
where $\vx_k^i$ and $\hat{\vx}_k^i$ are the $ith$ true state and state estimate at time $k$, respectively, and $m$ is the number of the Monte Carlo runs.
When the underlying probability density functions of noises are known, we use the particle filter in \cite{Arulampalam-Maskell-Gordon-Clapp02}, which is denoted by \textbf{PF-T}.
When the underlying probability density functions of noises are unknown, we denote \textbf{PF-G} for the particle filter where the state noise and measurement noise are assumed the truncated Gaussian noise with zero mean.
 At the same time, we may assume that the noises are uniform density functions, then we still use particle filter, which is denoted by \textbf{PF-U}.
The extended set-membership filter in \cite{Scholte-Campbell03} is denoted by \textbf{ESMF}. These four filters have the same initial bounding ellipsoid in this example.


The following simulation results are under Matlab R2012a with YALMIP.

Figures \ref{fig_03}-\ref{fig_05} present a comparison of the error bounds along position and velocity direction of MCSMF with those of PF-T, PF-G, PF-U and ESMF, respectively.  Figures \ref{fig_03}-\ref{fig_05} show that when the probability density functions of noises are known, the performance of the particle filter is better than that of MCSMF and ESMF. The reason may be that more information of the probability density of noises is used. However, when it is unknown, the performance of the particle filter is worse than that of  MCSMF. In addition, the figures also show that performance of ESMF is unstable. The reason may be that there are some uncertain parameters to be used in ESMF and the remainder is bounded by interval mathematics method, which is conservative and leads a bigger bounding ellipsoid than MCSMF.

Figures \ref{fig_06}-\ref{fig_07} present the target tracking trajectories along $\vx$ direction by MCSMF and PF-T, respectively. The bounds of MCSMF and the $3\sigma$ confidence bounds of PF-T are also plotted.  Figures \ref{fig_06}-\ref{fig_07} show that  the $3\sigma$ confidence bounds of particle filter is indeed tighter than that of MCSMF, but it cannot contain the true state at some time step. It is an too optimistic bound. However, the bounds of MCSMF do guarantee the containment of the true state at each time step. This is useful in some applications. For example, in a civilian
air traffic control system, the confidence bounds of trajectories  can be used to check the standard separation between pairs of
targets for maintenance of safety conditions (collision avoidance) and regularity of traffic flow in  \cite{Mazor-Averbuch-BarShalom-Dayan98}.

The CPU times of MCSMF, PF-T, PF-U and PF-G are plotted as a function of number of samples and particles in Figure \ref{fig_04}, respectively. It shows that CPU times of the three filters are increasing as the number of samples and particles is increasing. The magnitude of the CPU time of the three filters are similar.


%
%

\begin{figure}[h]
\vbox to 6cm{\vfill \hbox to \hsize{\hfill
\scalebox{0.6}[0.59]{\includegraphics{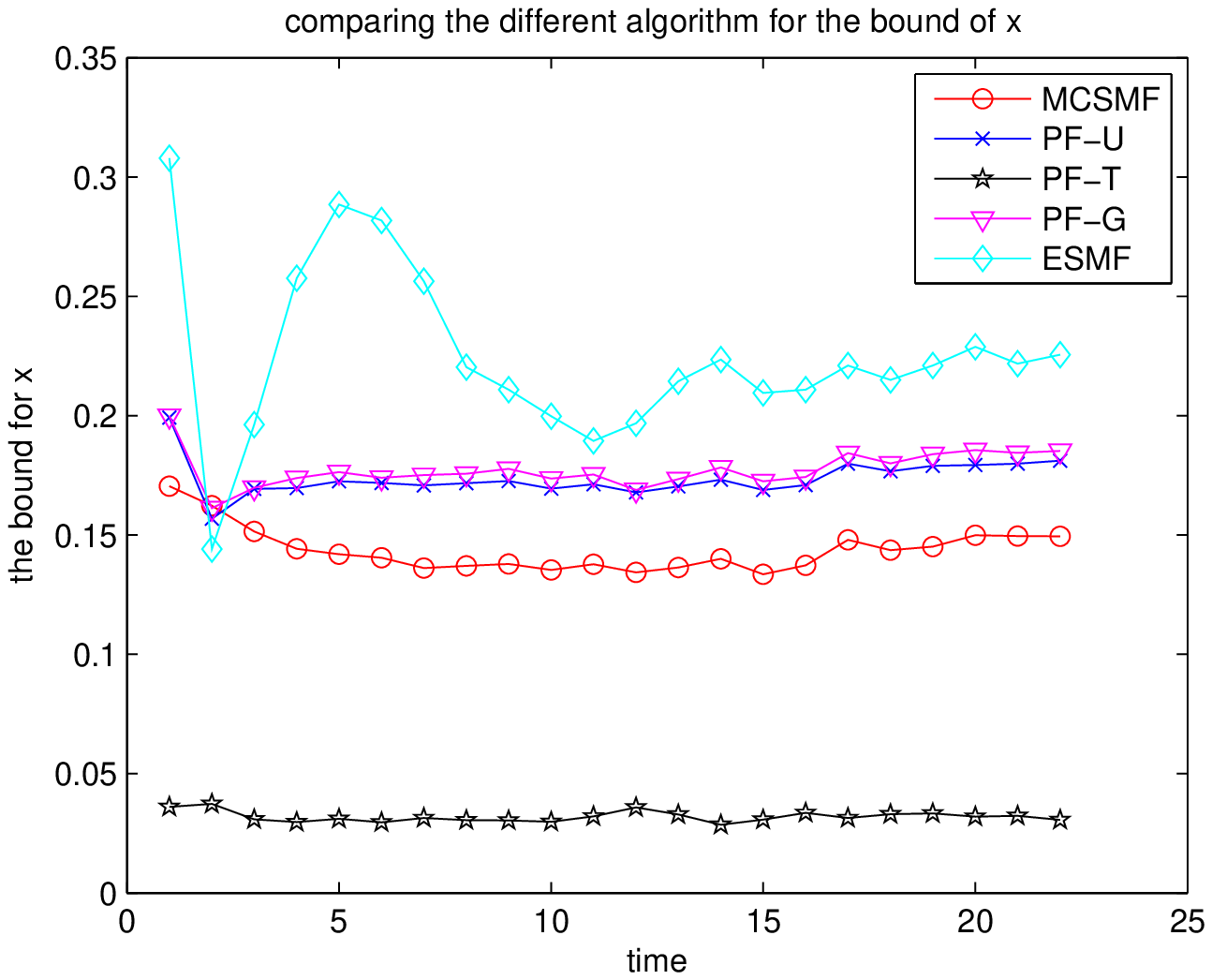}} \hfill}\vfill}
\caption{Comparison of the error bounds along position $\vx$ direction based on 200
Monte Carlo runs.}\label{fig_03}
\end{figure}

\begin{figure}[h]
\vbox to 6cm{\vfill \hbox to \hsize{\hfill
\scalebox{0.6}[0.59]{\includegraphics{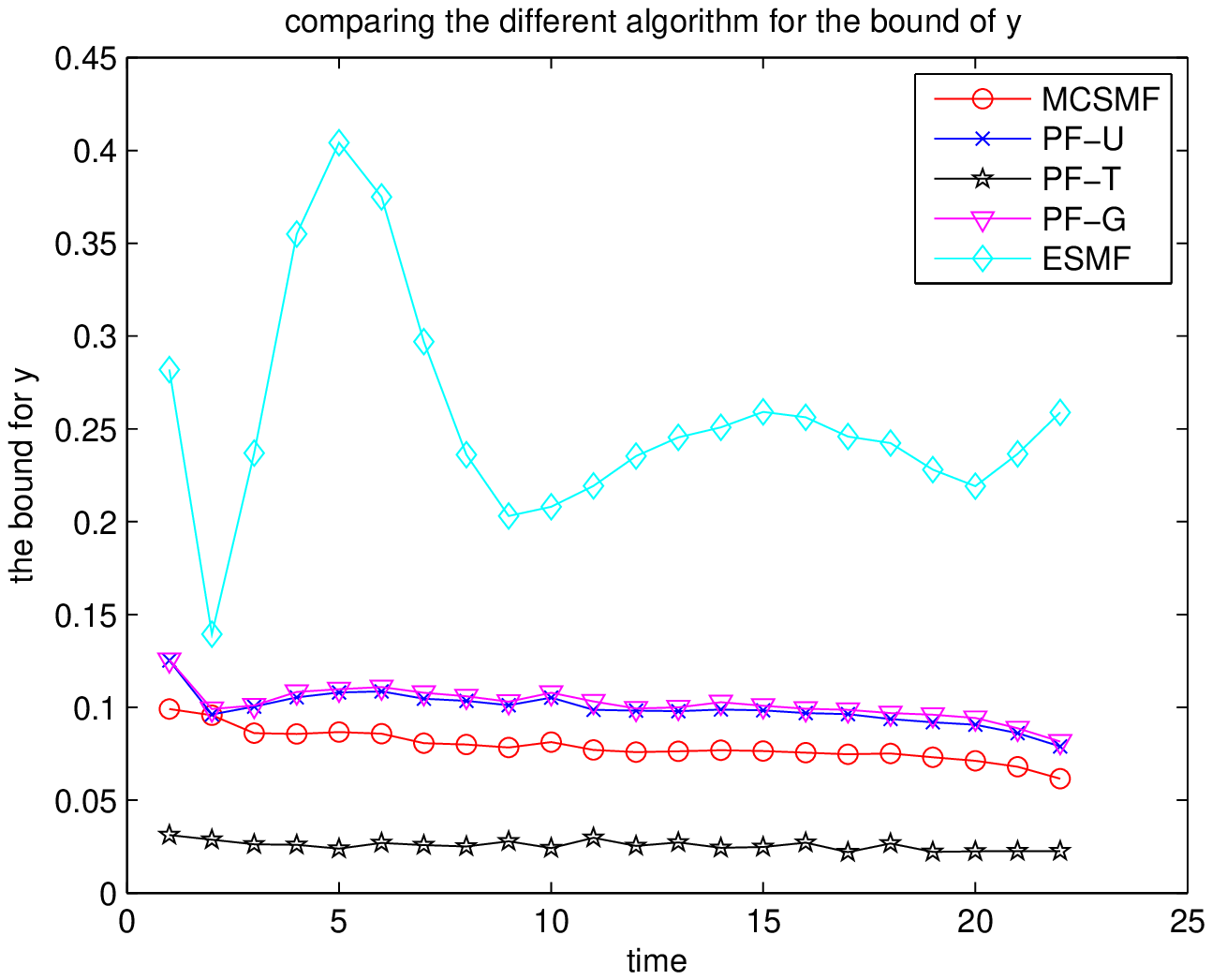}} \hfill}\vfill}
\caption{Comparison of the error bounds along position $\vy$ direction based on 200
Monte Carlo runs.}\label{fig_05}
\end{figure}


\begin{figure}[h]
\vbox to 6cm{\vfill \hbox to \hsize{\hfill
\scalebox{0.6}[0.59]{\includegraphics{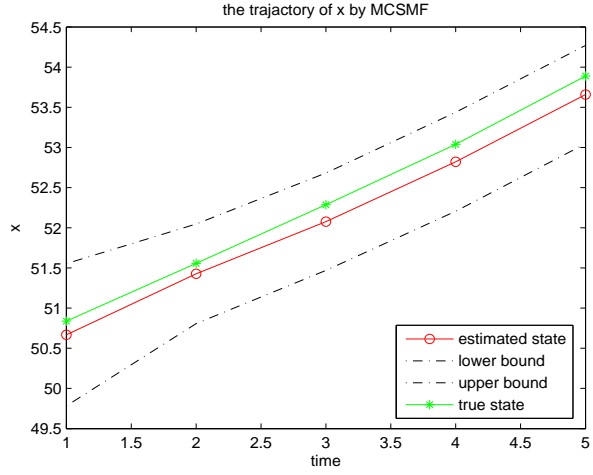}} \hfill}\vfill}
\caption{The target's  trajectory along $\vx$ direction by MCSMF }\label{fig_06}
\end{figure}
\begin{figure}[h]
\vbox to 6cm{\vfill \hbox to \hsize{\hfill
\scalebox{0.6}[0.59]{\includegraphics{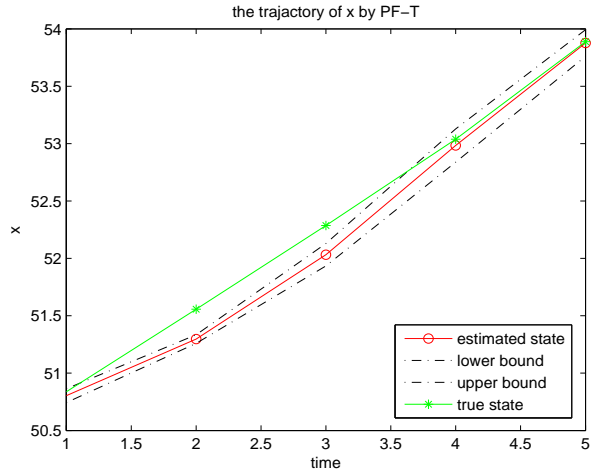}} \hfill}\vfill}
\caption{The target's  trajectory along $\vx$ direction by PF-T }\label{fig_07}
\end{figure}

\begin{figure}[h]
\vbox to 6cm{\vfill \hbox to \hsize{\hfill
\scalebox{0.6}[0.59]{\includegraphics{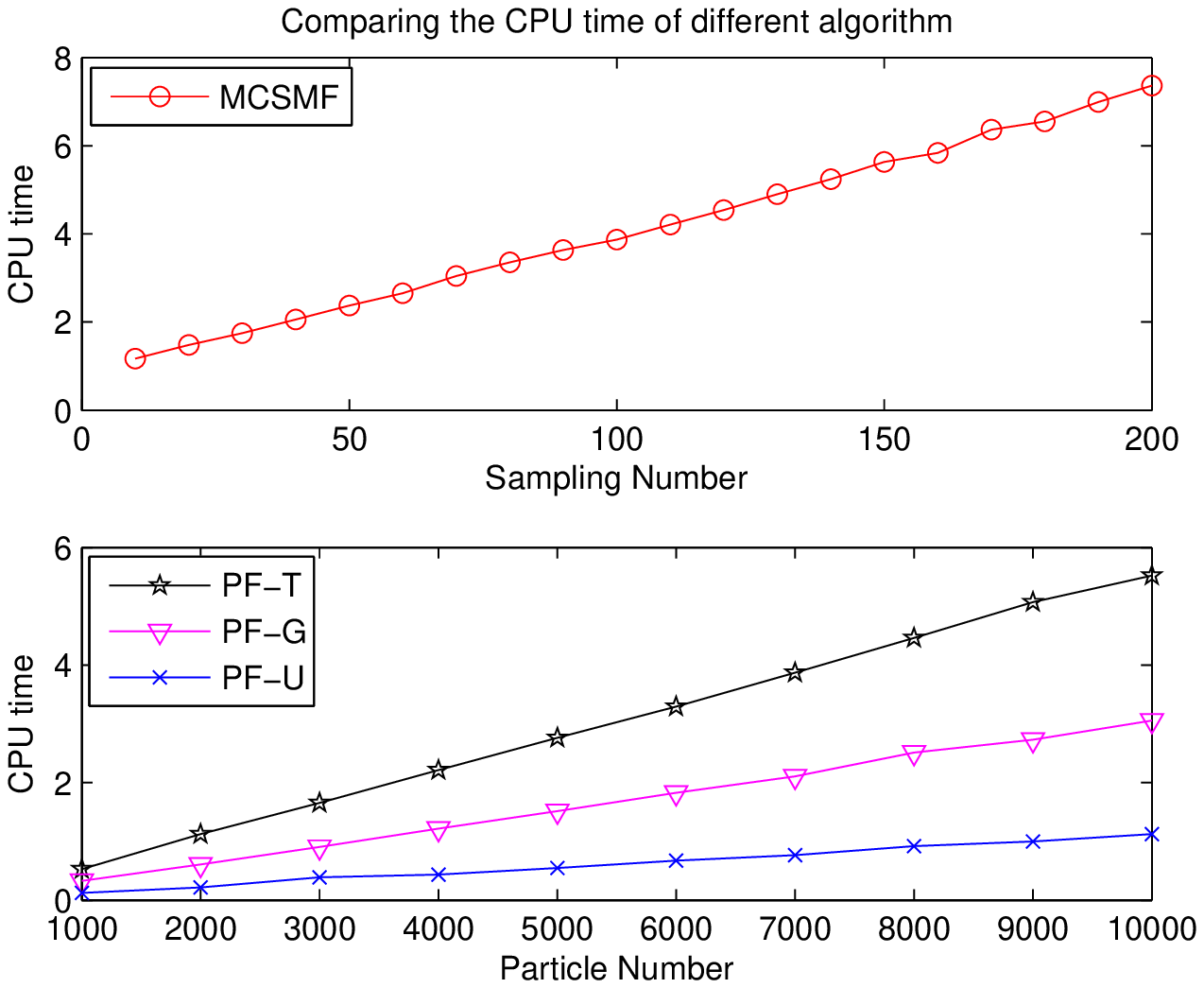}} \hfill}\vfill}
\caption{(up) The CPU times for MCSMF with different sampling numbers from the boundary. (bottom) The CPU times for PF-T, PF-U and PF-G with different particle numbers.}\label{fig_04}
\end{figure}


%

\section{Conclusion}\label{sec_7}
We have proposed a new class of filtering methods in bounded noise setting via set-membership theory
and Monte Carlo (boundary) sampling technique to determine a state estimation ellipsoid. The set-membership prediction and measurement update are derived by recent convex optimization methods based on S-procedure and Schur complement. To guarantee the on-line usage, the nonlinear dynamics are linearized about the current estimate and the remainder terms are then bounded by an ellipsoid, which can be written as a semi-infinite optimization problem. For a typical nonlinear dynamic system in target tracking, based on the remainder properties and the Inverse Function Theorem, the semi-infinite optimization problem can be efficiently solved by Monte Carlo boundary sampling technique. Numerical example shows that when the probability density functions of noises are unknown, the performance of MCSMF is better than that of the particle filter, and which is more robust than particle filter. Future work will involve, in the setting of MCSMF, the multi-sensor fusion, multiple target tracking and various applications
such as sensor management and placement for structures and different
types of wireless networks.

%

\section{APPENDIX}\label{sec_8}
\begin{lemma}\label{lem_1}\cite{Boyd-ElGhaoui-Feron-Balakrishnan94}
Let $\mF_0(\eta), \mF_1(\eta),\ldots, \mF_p(\eta)$, be quadratic functions in variable $\eta\in\mathcal {R}^{n}$
\begin{eqnarray}
\mF_i(\eta)=\eta^T\mT_i\eta, ~~i=0,\ldots, p
\end{eqnarray}
with $\mT_i=\mT_i^T$. Then the implication
\begin{eqnarray}
\mF_1(\eta)\leq0,\ldots,\mF_p(\eta)\leq0\Rightarrow\mF_0(\eta)\leq0
\end{eqnarray}
holds if there exist $\tau_1,\ldots,\tau_p\geq0$ such that
\begin{eqnarray}
\mT_0-\sum_{i=1}^p\tau_i\mT_i\preceq0.
\end{eqnarray}
\end{lemma}

\begin{lemma}\label{lem_2}
Schur Complements \cite{Boyd-ElGhaoui-Feron-Balakrishnan94}: Given constant matrices $\mA$, $\mB$, $\mC$, where $\mC=\mC^T$ and $\mA=\mA^T<0$, then
\begin{eqnarray}
\mC-\mB^T\mA^{-1}\mB\preceq0
\end{eqnarray}
if and only if
\begin{eqnarray}
\left[
  \begin{array}{cc}
    \mA & \mB \\
    \mB^T & \mC \\
  \end{array}
\right]\preceq0
\end{eqnarray}
or equivalently
\begin{eqnarray}
\left[
  \begin{array}{cc}
    \mC & \mB^T \\
    \mB & \mA \\
  \end{array}
\right]\preceq0
\end{eqnarray}
\end{lemma}


\begin{proof}[Proof of Theorem \ref{thm_3}]: Note that $\vx_k\in\mathcal {E}_{k}$ is equivalent to $\vx_k=\hat{\vx}_{k}+\mE_{k}\vu_{k}$, $\parallel \vu_{k}\parallel\leq 1$,
where $\mE_{k}$ is a Cholesky factorization of $\mP_{k}$;  By the equations (\ref{Eqpre_1}) and (\ref{Eqpre_3}),
\begin{eqnarray}
\nonumber \vx_{k+1}-\hat{\vx}_{k+1|k}&=& f_k(\vx_k)+\vw_k-\hat{\vx}_{k+1|k}\\[3mm]
\nonumber  &=&f_k(\hat{\vx}_{k}+\mE_{k}\vu_{k})+\vw_k-\hat{\vx}_{k+1|k}\\[3mm]
\label{Eqpre_108} &=&f_k(\hat{\vx}_{k})+\mJ_{f_k}\mE_{k}\vu_{k}+\ve_{f_k}+\mB_{f_k}\Delta_{f_k}+\vw_k-\hat{\vx}_{k+1|k}
\end{eqnarray}
and by the equations (\ref{Eqpre_2}) and (\ref{Eqpre_4})
\begin{eqnarray}
\nonumber \vy_{k}&=& h_{k}(\vx_{k})+\vv_{k}\\[3mm]
\label{Eqpre_109}    &=&h_k(\hat{\vx}_{k})+\mJ_{h_k}\mE_{k}\vu_{k}+\ve_{h_k}+\mB_{h_k}\Delta_{h_k}+\vv_k
\end{eqnarray}

If we denote by
\begin{eqnarray}
\label{Eqpre_110} \xi=[1, ~\vu_{k}^T, ~\vw_k^T, ~\vv_{k}^T, ~\Delta_{f_k}^T, ~\Delta_{h_{k}}^T]^T,
\end{eqnarray}
then (\ref{Eqpre_108}) and (\ref{Eqpre_109}) can be rewritten as
\begin{eqnarray}
\label{Eqpre_111} \vx_{k+1}-\hat{\vx}_{k+1|k}&=&  \Phi_{k+1|k}(\hat{\vx}_{k+1|k})\xi\\[3mm]
\label{Eqpre_112} 0&=&  \Psi_{k+1|k}(\vy_{k})\xi,
\end{eqnarray}
where $\Phi_{k+1|k}(\hat{\vx}_{k+1|k})$ and $\Psi_{k+1|k}(\vy_{k})$ are denoted by (\ref{Eqpre_105}) and (\ref{Eqpre_106}), respectively.

Moreover, the condition that $\vx_{k+1}\in\mathcal {E}_{k+1|k}$ whenever I) $\vx_k$ is in $\mathcal {E}_{k}$
II) the process and measurement noises $\vw_k, \vv_{k}$ are bounded
in ellipsoidal sets, i.e., $\vw_k\in\mW_k$, $\vv_{k}\in\mV_{k}$ is equivalent to
\begin{eqnarray}
\label{Eqpre_113}  \xi^T\Phi_{k+1|k}(\hat{\vx}_{k+1|k})^T(\mP_{k+1|k})^{-1}\Phi_{k+1|k}(\hat{\vx}_{k+1|k})\xi\leq1,
\end{eqnarray}
whenever
\begin{eqnarray}
\label{Eqpre_114}  \parallel \vu_{k}\parallel&\leq& 1 ,\\[3mm]
\label{Eqpre_115}  \vw_k^T\mQ_k^{-1}\vw_k&\leq& 1,\\[3mm]
\label{Eqpre_116} \vv_{k}^T\mR_k^{-1}\vv_{k}&\leq& 1,\\[3mm]
\label{Eqpre_117}\parallel \Delta_{f_k}\parallel&\leq &1,\\[3mm]
\label{Eqpre_118}\parallel \Delta_{h_{k}}\parallel&\leq &1.
\end{eqnarray}
The equations (\ref{Eqpre_114})--(\ref{Eqpre_118}) is equivalent to
\begin{eqnarray}
\label{Eqpre_119}  \xi^T\diag(-1,I,0,0,0,0)\xi&\leq& 0,\\[3mm]
\label{Eqpre_120}  \xi^T\diag(-1,0,\mQ_k^{-1},0,0,0)\xi&\leq& 0,\\[3mm]
\label{Eqpre_121}  \xi^T\diag(-1,0,0,\mR_k^{-1},0,0)\xi&\leq& 0,\\[3mm]
\label{Eqpre_122}  \xi^T\diag(-1,0,0,0,I,0)\xi&\leq& 0,\\[3mm]
\label{Eqpre_123}  \xi^T\diag(-1,0,0,0,0,I)\xi&\leq& 0.
\end{eqnarray}
where $I$ and $0$ are matrices with compatible dimensions.

By  $\mathcal {S}$-procedure Lemma \ref{lem_1} and Eq.
(\ref{Eqpre_112}), a sufficient condition such that the inequalities
(\ref{Eqpre_119})-(\ref{Eqpre_123}) imply (\ref{Eqpre_113}) to hold is that
there exist scalars $\tau^y$ and nonnegative scalars
$\tau^u\geq0, \tau^w\geq0,
\tau^v\geq0,\tau^f\geq0,\tau^h\geq0$, such that
\begin{eqnarray}
\nonumber&& \Phi_{k+1|k}(\hat{\vx}_{k+1|k})^T(\mP_{k+1|k})^{-1}\Phi_{k+1|k}(\hat{\vx}_{k+1|k})\\[3mm]
\nonumber&&-\diag(1,0,0,0,0,0,0)\\[3mm]
\nonumber&&-\tau^u\diag(-1,I,0,0,0,0,0)\\[3mm]
\nonumber&&-\tau^w\diag(-1,0,\mQ_k^{-1},0,0,0,0)\\[3mm]
\nonumber&&-\tau^v\diag(-1,0,0,\mR_k^{-1},0,0,0)\\[3mm]
\nonumber&&-\tau^f\diag(-1,0,0,0,I,0,0)\\[3mm]
\nonumber&&-\tau^h\diag(-1,0,0,0,0,0,I)\\[3mm]
\label{Eqpre_124}&&-\tau^y \Psi_{k+1|k}(\vy_{k})^T \Psi_{k+1|k}(\vy_{k}) \preceq0
\end{eqnarray}
Furthermore, (\ref{Eqpre_124}) is
written in the following compact form:
\begin{eqnarray}
\label{Eqpre_125} \Phi_{k+1|k}(\hat{\vx}_{k+1|k})^T(\mP_{k+1|k})^{-1}\Phi_{k+1|k}(\hat{\vx}_{k+1|k})-\Xi-\tau^y \Psi_{k+1|k}(\vy_{k})^T \Psi_{k+1|k}(\vy_{k}) \preceq0
\end{eqnarray}
where $\Xi$ is denoted by (\ref{Eqpre_107}).

If we denote
$(\Psi_{k+1|k}(\vy_{k}))_{\bot}$ is the orthogonal complement of $\Psi_{k+1|k}(\vy_{k})$, then
(\ref{Eqpre_125}) is equivalent to
\begin{eqnarray}
\label{Eqpre_126}  &&((\Psi_{k+1|k}(\vy_{k}))_{\bot})^T\Phi_{k+1|k}(\hat{\vx}_{k+1|k})^T(\mP_{k+1|k})^{-1}\Phi_{k+1|k}(\hat{\vx}_{k+1|k})(\Psi_{k+1|k}
(\vy_{k}))_{\bot}\\
\nonumber&&~~~-((\Psi_{k+1|k}(\vy_{k}))_{\bot})^T\Xi(\Psi_{k+1|k}(\vy_{k}))_{\bot} \preceq0
\end{eqnarray}
Using Schur complements, (\ref{Eqpre_126}) is
equivalent to
\begin{eqnarray}
\label{Eqpre_127}&&\left[\begin{array}{cc}
    -\mP_{k+1|k}&\Phi_{k+1|k}(\hat{\vx}_{k+1|k})(\Psi_{k+1|k}(\vy_{k}))_{\bot}\\[3mm]
    (\Phi_{k+1|k}(\hat{\vx}_{k+1|k})(\Psi_{k+1|k}(\vy_{k}))_{\bot})^T& ~~-(\Psi_{k+1|k}(\vy_{k}))_{\bot}^T\Xi(\Psi_{k+1|k}(\vy_{k}))_{\bot}\\
\end{array}\right]\preceq0,\\[3mm]
\label{Eqpre_128}&&-\mP_{k+1|k}\prec0.
\end{eqnarray}

Therefore, if
$\hat{x}_{k+1|k}$, $\mP_{k+1|k}$ satisfy (\ref{Eqpre_127}) and
(\ref{Eqpre_128}), then the state $x_{k+1}$ belongs to $\mathcal
{E}_{k+1|k}$, whenever I) $\vx_k$ is in $\mathcal {E}_{k}$,
II) the process and measurement noises $\vw_k, \vv_{k}$ are bounded
in ellipsoidal sets, i.e., $\vw_k\in\mW_k$, $\vv_{k}\in\mV_{k}$.

Summarizing  the above results,  the computation of the
predicted  bounding ellipsoid by minimizing a
size measure $f(\mP_{k+1|k})$ (\ref{Eqpre_101}) is  Theorem \ref{thm_3}.
\end{proof}

\begin{proof}[Proof of Theorem \ref{thm_4}]: Note that we have get $\vx_{k+1}\in\mathcal {E}_{k+1|k}$ in prediction step, which is equivalent to $\vx_{k+1}=\hat{\vx}_{k+1|k}+\mE_{k+1|k}\vu_{k+1|k}$, $\parallel \vu_{k+1|k}\parallel\leq 1$,
where $\mE_{k+1|k}$ is a Cholesky factorization of $\mP_{k+1|k}$, then,
\begin{eqnarray}
\label{Eqpre_136}  \vx_{k+1}-\hat{\vx}_{k+1}&=& \hat{\vx}_{k+1|k}+\mE_{k+1|k}\vu_{k+1|k}-\hat{\vx}_{k+1}
\end{eqnarray}
and by the equations (\ref{Eqpre_2}) and (\ref{Eqpre_4})
\begin{eqnarray}
\nonumber \vy_{k+1}&=& h_{k+1}(\vx_{k+1})+\vv_{k+1}\\
\label{Eqpre_137}   &=&h_k(\hat{\vx}_{k+1|k})+\mJ_{h_{k+1}}\mE_{k+1|k}\vu_{k+1|k}+\ve_{h_k+1}+\mB_{h_{k+1}}\Delta_{h_{k+1}}+\vv_{k+1}
\end{eqnarray}
If we denote by
\begin{eqnarray}
\label{Eqpre_138} \xi=[1, ~\vu_{k+1|k}^T, ~\vv_{k+1}^T, ~\Delta_{h_{k+1}}^T]^T,
\end{eqnarray}
then (\ref{Eqpre_136}) and (\ref{Eqpre_137}) can be rewritten as
\begin{eqnarray}
\label{Eqpre_139} \vx_{k+1}-\hat{\vx}_{k+1}&=&  \Phi_{k+1}(\hat{\vx}_{k+1})\xi\\[3mm]
\label{Eqpre_140} 0&=&  \Psi_{k+1}(\vy_{k+1})\xi,
\end{eqnarray}
where $\Phi_{k+1}(\hat{\vx}_{k+1})$ and $\Psi_{k+1}(\vy_{k+1})$ are denoted by (\ref{Eqpre_133}) and (\ref{Eqpre_134}), respectively.

Moreover, the condition that $\vx_{k+1}\in\mathcal {E}_{k+1}$ whenever I) $\vx_{k+1}$ is in $\mathcal {E}_{k+1|k}$
II) measurement noises $\vv_{k+1}$ are bounded
in ellipsoidal sets, i.e., $\vv_{k+1}\in\mV_{k+1}$ is equivalent to
\begin{eqnarray}
\label{Eqpre_141}  \xi^T\Phi_{k+1}(\hat{\vx}_{k+1})^T(\mP_{k+1})^{-1}\Phi_{k+1}(\hat{\vx}_{k+1})\xi\leq1,
\end{eqnarray}
whenever
\begin{eqnarray}
\label{Eqpre_142}  \parallel \vu_{k}\parallel&\leq& 1 ,\\[3mm]
\label{Eqpre_143} \vv_{k+1}^T\mR_{k+1}^{-1}\vv_{k+1}&\leq& 1,\\[3mm]
\label{Eqpre_144}\parallel \Delta_{h_{k+1}}\parallel&\leq &1.
\end{eqnarray}
The equations (\ref{Eqpre_142})--(\ref{Eqpre_144}) is equivalent to
\begin{eqnarray}
\label{Eqpre_145}  \xi^T\diag(-1,I,0,0)\xi&\leq& 0,\\[3mm]
\label{Eqpre_146}  \xi^T\diag(-1,0,\mR_{k+1}^{-1},0)\xi&\leq& 0,\\[3mm]
\label{Eqpre_147}  \xi^T\diag(-1,0,0,I)\xi&\leq& 0,
\end{eqnarray}
where $I$ and $0$ are matrices with compatible dimensions.

By  $\mathcal {S}$-procedure Lemma \ref{lem_1} and Eq.
(\ref{Eqpre_140}), a sufficient condition such that the inequalities
(\ref{Eqpre_145})-(\ref{Eqpre_147}) imply (\ref{Eqpre_141}) to hold is that
there exist scalars $\tau^y$ and nonnegative scalars
$\tau^u\geq0, \tau^v\geq0, \tau^h\geq0$, such that
\begin{eqnarray}
\nonumber&& \Phi_{k+1}(\hat{\vx}_{k+1})^T(\mP_{k+1})^{-1}\Phi_{k+1}(\hat{\vx}_{k+1})\\[3mm]
\nonumber&&-\diag(1,0,0,0,0)\\[3mm]
\nonumber&&-\tau^u\diag(-1,I,0,0,0)\\[3mm]
\nonumber&&-\tau^v\diag(-1,0,0,\mR_{k+1}^{-1},0)\\[3mm]
\nonumber&&-\tau^f\diag(-1,0,0,0,I)\\[3mm]
\label{Eqpre_190}&&-\tau^y \Psi_{k+1}(\vy_{k+1})^T \Psi_{k+1}(\vy_{k+1}) \preceq0
\end{eqnarray}
Furthermore, (\ref{Eqpre_190}) is
written in the following compact form:
\begin{eqnarray}
\label{Eqpre_148} \Phi_{k+1}(\hat{\vx}_{k+1})^T(\mP_{k+1})^{-1}\Phi_{k+1}(\hat{\vx}_{k+1})-\Xi-\tau^y \Psi_{k+1}(\vy_{k+1})^T \Psi_{k+1}(\vy_{k+1}) \preceq0
\end{eqnarray}
where $\Xi$ is denoted by (\ref{Eqpre_135}).

If we denote
$(\Psi_{k+1}(\vy_{k+1}))_{\bot}$ is the orthogonal complement of $\Psi_{k+1}(\vy_{k+1})$, then
(\ref{Eqpre_148}) is equivalent to
\begin{eqnarray}
\nonumber && ((\Psi_{k+1}(\vy_{k+1}))_{\bot})^T\Phi_{k+1}(\hat{\vx}_{k+1})^T(\mP_{k+1})^{-1}\Phi_{k+1}(\hat{\vx}_{k+1})(\Psi_{k+1}(\vy_{k}))_{\bot}\\
\label{Eqpre_149}&&-((\Psi_{k+1}(\vy_{k+1}))_{\bot})^T\Xi(\Psi_{k+1}(\vy_{k+1}))_{\bot} \preceq0
\end{eqnarray}
Using Schur complements Lemma \ref{lem_2}, (\ref{Eqpre_149}) is
equivalent to
\begin{eqnarray}
\label{Eqpre_150}&&\left[\begin{array}{cc}
    -\mP_{k+1}&\Phi_{k+1}(\hat{\vx}_{k+1})(\Psi_{k+1}(\vy_{k+1}))_{\bot}\\[3mm]
    (\Phi_{k+1}(\hat{\vx}_{k+1})(\Psi_{k+1}(\vy_{k+1}))_{\bot})^T& ~~-(\Psi_{k+1}(\vy_{k+1}))_{\bot}^T\Xi(\Psi_{k+1}(\vy_{k+1}))_{\bot}\\
\end{array}\right]\preceq0,\\[3mm]
\label{Eqpre_151}&&-\mP_{k+1}\prec0.
\end{eqnarray}

Therefore, if
$\hat{x}_{k+1}$, $\mP_{k+1}$ satisfy (\ref{Eqpre_150}) and
(\ref{Eqpre_151}), then the state $x_{k+1}$ belongs to $\mathcal
{E}_{k+1}$, whenever I) $\vx_{k+1}$ is in $\mathcal {E}_{k+1|k}$,
II) measurement noises $ \vv_{k+1}$ are bounded
in ellipsoidal sets, i.e., $\vv_{k+1}\in\mV_{k+1}$.

Summarizing  the above results,  the computation of the
measurement update  bounding ellipsoid by minimizing a
size measure $f(\mP_{k+1})$ (\ref{Eqpre_125}) is  Theorem \ref{thm_4}.
\end{proof}


\begin{thebibliography}{10}

\bibitem{Kalman60}
R.~E. Kalman, ``A new approach to linear filtering and prediction problems,''
  {\em Transactions of ASME, Journal of Basic Engineering}, vol.~82, 1960.

\bibitem{BarShalom-Li-Kirubarajan01}
Y.~Bar-Shalom, X.~Li, and T.~Kirubarajan, {\em Estimation with Applications to
  Tracking and Navigation}.
\newblock New York: Wiley, 2001.

\bibitem{Simon06}
D.~Simon, {\em Optimal State Estimation: {Kalman}, $H_\infty$, and Nonlinear
  Approaches}.
\newblock Wiley-Interscience, 2006.

\bibitem{Gordon-Salmond-Smith93}
N.~J. Gordon, D.~J. Salmond, and A.~F.~M. Smith, ``Novel approach to
  {nonlinear/non-Gaussian Bayesian} state estimation,'' {\em IEE Proceedings
  F--Radar and Signal Processing}, vol.~140, pp.~107--113, April 1993.

\bibitem{Kong-Liu-Wong94}
A.~Kong, J.~S. Liu, and W.~H. Wong, ``Sequential imputations and {Bayesian}
  missing data problems,'' {\em Journal of the American Statistical
  Association}, vol.~89, pp.~278--288, 1994.

\bibitem{Liu-Chen98}
J.~S. Liu and R.~Chen, ``Sequential {Monte Carlo} methods for dynamic
  systems,'' {\em Journal of the American Statistical Association}, vol.~93,
  pp.~1032--1044, 1998.

\bibitem{Kotecha-Djuric03}
J.~H. Kotecha and P.~M. Djuric, ``Gaussian particle filtering,'' {\em IEEE
  Transactions on Signal Processing}, vol.~51, no.~10, pp.~2592--2601, 2003.

\bibitem{Crisan-Doucet02}
D.~Crisan and A.~Doucet, ``A survey of convergence results on particle
  filtering methods for practitioners,'' {\em IEEE Transactions on Signal
  Processing}, vol.~50, no.~3, pp.~736--746, 2002.

\bibitem{Arulampalam-Maskell-Gordon-Clapp02}
M.~S. Arulampalam, S.~Maskell, N.~Gordon, and T.~Clapp, ``A tutorial on
  particle filters for online nonlinear/non-{Gaussian} {Bayesian} tracking,''
  {\em IEEE Transactions on Signal Processing}, vol.~50, pp.~174--188, February
  2002.

\bibitem{Chen-Wang-Liu00}
R.~Chen, X.~Wang, and J.~S. Liu, ``Adaptive joint detection and decoding in
  flat-fading channels via mixture {Kalman} filtering,'' {\em IEEE Transaction
  on Information Theory}, vol.~46, pp.~2079--2094, September 2000.

\bibitem{Zheng-Niu-Varshney14}
Y.~Zheng, R.~Niu, and P.~K. Varshney, ``Sequential {Bayesian} estimation with
  censored data for multi-sensor systems,'' {\em IEEE Transactions on Signal
  Processing}, vol.~62, pp.~2626--2641, May 2014.

\bibitem{Polyak-Nazin-Durieu-Walter04}
B.~T. Polyak, S.~A. Nazin, C.~Durieu, and E.~Walter, ``Ellipsoidal parameter or
  state estimation under model uncertainty,'' {\em Automatica}, vol.~40,
  pp.~1171--1179, 2004.

\bibitem{Schweppe68}
F.~C. Schweppe, ``Recursive state estimation: Unknown but bounded errors and
  system inputs,'' {\em IEEE Transactions on Automatic Control}, vol.~AC-13,
  pp.~22--28, February 1968.

\bibitem{Bertsekas-Rhodes71}
D.~P. Bertsekas and I.~B. Rhodes, ``Recursive state estimation for a
  setmembership description of uncertainty,'' {\em IEEE Transactions on
  Automatic Control}, vol.~16, pp.~117--128, February 1971.

\bibitem{Durieu-Walter-Polyak01}
C.~Durieu, E.~Walter, and B.~T. Polyak, ``Multi-input multi-output ellipsoidal
  state bounding,'' {\em Journal of Optimization Theory and Applications,},
  vol.~111, no.~2, pp.~273--303, 2001.

\bibitem{Calafiore-ElGhaoui04}
G.~Calafiore and L.~{El Ghaoui}, ``Ellipsoidal bounds for uncertain equations
  and dynamical systems,'' {\em Automatica}, vol.~40, pp.~773--787, 2004.

\bibitem{Shen-Zhu-Song-Luo11}
X.~Shen, Y.~Zhu, E.~Song, and Y.~Luo, ``Minimizing {Euclidian} state estimation
  error for linear uncertain dynamic systems based on multisensor and
  multi-algorithm fusion,'' {\em IEEE Transactions on Information Theroy},
  vol.~57, pp.~7131--7146, October 2011.

\bibitem{Jaulin-Kieffer-Didrit-Walter01}
L.~Jaulin, M.~Kieffer, O.~Didrit, and E.~Walter, {\em Applied Interval
  Analysis}.
\newblock Springer, 2001.

\bibitem{Shamma-Tu97}
J.~S. Shamma and K.~Tu, ``Approximate set-valued observers for nonlinear
  systems,'' {\em IEEE Transactions on Automatic Control}, vol.~42,
  pp.~648--658, May 1997.

\bibitem{Yang-Li10}
F.~Yang and Y.~Li, ``Set-membership fuzzy filtering for nonlinear discrete-time
  systems,'' {\em IEEE Transaction on Systems, Man, And Cybernetics-Part B:
  Cybernetics}, vol.~40, pp.~116--124, February 2010.

\bibitem{Morrell-Stirlling03}
D.~R. Morrell and W.~C. Stirlling, ``An extended set-valued {Kalman} filter,''
  {\em Proceeding of ISIPTA}, pp.~396--407, 2003.

\bibitem{Wei-Wang-Shen10}
G.~Wei, Z.~Wang, and B.~Shen, ``Error-constrained filtering for a class of
  nonlinear time-varying delay systemswith non-gaussian noises,'' {\em IEEE
  Transaction on Automatic Control}, vol.~55, pp.~2876--2882, December 2010.

\bibitem{Scholte-Campbell03}
E.~Scholte and M.~E. Campbell, ``A nonlinear set-membership filter for on-line
  applications,'' {\em International Journal of Robust and Nonlinear Control},
  vol.~13, pp.~1337--1358, December 2003.

\bibitem{Moore66}
R.~E. Moore, {\em Interval Analysis}.
\newblock Prentice-Hall: Englewood Cliffs, NJ, 1966.

\bibitem{ElGhaoui-Calafiore01}
L.~{El Ghaoui} and G.~Calafiore, ``Robust filtering for discrete-time systems
  with bounded noise and parametric uncertainty,'' {\em IEEE Transactions on
  Automatic Control}, vol.~46, no.~7, pp.~1084--1089, 2001.

\bibitem{Nesterov-Nemirovski94}
Y.Nesterov and A.Nemirovski, ``Interior point polynomial methods in convex
  programming: Theroy and applications,'' {\em Philadelphia, PA: SIAM}, 1994.

\bibitem{Lofberg04}
J.~L\"{o}fberg, ``{YALMIP}: a toolbox for modelling and optimization in
  {Matlab},'' in {\em Proceedings of the IEEE CACSD Symposium}, (Taipei,
  Taiwan), pp.~284--289, September 2004.

\bibitem{Sturm99}
J.~F. Sturm, ``Using {SeDuMi} 1.02, a {Matlab} toolbox for optimization over
  symmetric cones,'' {\em Optimization Methods and Software}, vol.~11,
  pp.~625--653, 1999.

\bibitem{Boyd-Vandenberghe04}
S.~Boyd and L.~Vandenberghe, {\em Convex Optimization}.
\newblock Cambridge University Press, 2004.

\bibitem{Vandenberghe-Boyd96}
L.Vandenberghe and S.Boyd, ``Semidefinite programming,'' {\em SIAM Review},
  vol.~38, pp.~49--95, March 1996.

\bibitem{Selin15}
S.~D. Ahipa$\c{s}$ao$\breve{\emph{g}}$lu, ``A first-order algorithm for the
  a-optimal experimental design problem: a mathematical programming approach,''
  {\em Statistics and Computing}, vol.~25, pp.~1113--1127, 2015.

\bibitem{Spivak65}
M.~Spivak, {\em Calculus on manifolds}.
\newblock Benjamin, New York, 1965.

\bibitem{Rosenlicht68}
M.~Rosenlicht, {\em Introduction to analysis}.
\newblock Glenview, III. : Scott, Foresman.

\bibitem{Mazor-Averbuch-BarShalom-Dayan98}
E.~Mazor, A.~Averbuch, Y.~Bar-Shalom, and J.~Dayan, ``Interacting multiple
  model methods in target tracking: {A} survey,'' {\em IEEE Transactions on
  Aerospace and Electronic Systems}, vol.~34, pp.~103--123, JANUARY 1998.

\bibitem{Boyd-ElGhaoui-Feron-Balakrishnan94}
S.~Boyd, L.~E. Ghaoui, E.~Feron, and V.~Balakrishnan, {\em Linear Matrix
  Inequalities in System and Control Theory}.
\newblock Philadelphia, PA: SIAM (Studies in Applied Mathematics), June 1994.

\end{thebibliography}
\end{document}